\documentclass[a4paper,10pt]{amsart}

\usepackage{amsmath,amsfonts,amssymb,amsthm,mathrsfs}
\usepackage{bm}
\usepackage{color}
\usepackage{esint}
\usepackage{graphicx}
\usepackage{graphics}
\usepackage{geometry}
\usepackage[all]{xy}

\usepackage{tikz}
\usepackage[bookmarks=false]{hyperref}

\renewcommand{\epsilon}{\varepsilon}

\newcommand{\N}{{\mathbb N}}

\newcommand{\R}{{\mathbb R}}
\newcommand{\C}{{\mathbb C}}

\def \asr {[\sqrt{a}]}

\newcommand{\erfc}{{\operatorname{erfc}}}

\renewcommand{\phi}{\varphi}

\newcommand{\hcal}{\mathcal{H}}

\newtheorem{theorem}{{Theorem}}[section]

\newtheorem{cor}[theorem]{{Corollary}}
\newtheorem{lem}[theorem]{{Lemma}}
\newtheorem{prop}[theorem]{{Proposition}}

\newenvironment{remark}{\medskip\noindent{\it Remark:\/} }{\medskip}

\newtheorem{conj}[theorem]{Conjecture}

\theoremstyle{definition}
\newtheorem{definition}[theorem]{Definition}

\numberwithin{equation}{section}

\def \N {\mathbb N}

\def \C {\mathbb C}

\def \R {\mathbb R}

\def \blambda {\bm{\lambda}}

\title[An infinite dimensional balanced embedding problem]{ An infinite dimensional balanced embedding problem II: uniqueness}

\author{Jingzhou Sun}
\thanks{The author is partially supported by NNSF of China no.11701353. }
\address{Department of Mathematics, Shantou University, Shantou City, Guangdong Province 515063, China}
\email{jzsun@stu.edu.cn}

\begin{document}
	\begin{abstract}
		This is the sequel to our first paper concerning the balanced embedding of a non-compact complex manifold into an infinite-dimensional projective space. We prove the uniqueness of such an embedding. The proof relies on fine estimates of the asymptotics of a balanced embedding.
	\end{abstract}
	
	\maketitle
	
	\tableofcontents

	\section{Introduction}
	
	In this article, we continue our study of a model problem concerning the balanced embedding of a non-compact complex manifold into a projective Hilbert space started in \cite{sun-sun-2}. 
	
	For the general set-up and motivations of this problem, the readers are referred to \cite{sun-sun-2} and the references there. The readers are also referred to \cite{RezaS-Duke}, \cite{Seyyedali2013}and \cite{Keller2014} etc. for the applications of balanced metric in algebraic geometry.
	For the statement of our results, we repeat the necessary notations.  
	
	Let $\mathcal H$ be a separable complex Hilbert space, and let $\mathbb P(\mathcal H)$ be the associated projective Hilbert space, viewed as a set. We denote by $\mathfrak h(\mathcal H)$ the space of bounded self-adjoint operators on $\mathcal H$. There exists a natural injective map 
	$$\iota: \mathbb P(\mathcal H)\rightarrow \mathfrak h(\mathcal H); [z]\mapsto \frac{z\otimes z^*}{|z|^2},$$
	where $z\otimes z^*$ is the operator defined by $y\mapsto \langle y,z\rangle z,y\in \hcal$.
	Let $X$ be a finite-dimensional (possibly non-compact) complex manifold with a volume form $d\mu$. We say a map $F: X\rightarrow \mathbb P(\mathcal H)$ is a holomorphic embedding if it is injective and locally induced by a holomorphic embedding into $\mathcal H$.
	
	We say a holomorphic embedding $F: X\rightarrow \mathcal H$ is  \emph{balanced} if it satisfies the equation
	\begin{equation}
	c(F)\equiv\int_{X} \iota(F(z)) d\mu=C\cdot\text{Id}, 
	\end{equation}
	where $C$ is a constant.
	More generally, we consider a pair $(X, D)$, where $D$ is a divisor in $X$. Let $d\mu$ be a volume form on $X$ and $d\nu$ be a volume form on $D$. We say that a holomorphic embedding $F: X\rightarrow \mathbb P(\mathcal H)$ is $(\beta, C)$-balanced for some $\beta, C\in \mathbb R$ if
	\begin{equation}\label{first balance equation} \int_X \iota(F(z)) d\mu+\beta \int_D \iota(F(z)) d\nu=C\cdot\text{Id},	
	\end{equation}
	where $\iota(F(z))$ denotes the K\"ahler form on $\mathbb P(\mathcal H)$ induced by the Fubini-Study metric.
	
	We consider the setting where $X$ is the complex plane $\mathbb C$ equipped with the standard volume form $d\mu = \frac{1}{\pi}\omega_0$, and $D$ is the origin ${0}$ equipped with the dirac measure $\delta_0$. Let $\mathcal H$ be the dual of the Hilbert space of entire holomorphic functions on $\mathbb C$ endowed with the $L^2$ inner product associated with $d\mu$. Then a basis of $\hcal^*$ induces a map $F:X \to \mathbb P(\mathcal H)$. And we are interested to know if we can find a basis so that the induced map $F$ is $(\beta,1)$-balanced, $\beta\in [0, 1)$.
	Then our main result is the following.
	
	\begin{theorem}\label{t:main}
		For $\beta\in [0, 1)$ there is a $(\beta,1)$-balanced embedding $F: X\rightarrow \mathbb P(\mathcal H)$ which is unique up to the action of $U(\mathcal H)$. 
	\end{theorem}
	Since the existence has been proved in \cite{sun-sun-2}, we only need to show the uniqueness part. As explained in \cite{sun-sun-2}, this is equivalent to  proving the uniqueness of the solution of the following equation
	\begin{equation}\label{e-2}
		\int_0^{\infty} \frac{f(sx)}{f(x)}dx=\frac{1}{1-s}-\beta, 
		\end{equation}
		for all $s\in[0,1)$, where $$f(x)=\sum_{i=0}^\infty e^{-\lambda_i}x^i,$$
	is an analytic function on $[0, \infty)$. As equation \ref{e-2} does not change when we multiply a constant on $f(x)$, we need to prove the uniqueness up to a scalar multiplication.
	
	In the case of $\beta=0$, the uniqueness is already non-trivial, and it follows from a result in analysis by Miles-Williamson \cite{MilesWill}. In \cite{loi-cn}, Cuccu-Loi proved the uniqueness of balanced metric($\beta=0$) on $\C^n$ under some symmetry condition, by using the result of Miles-Williamson \cite{MilesWill}. It is also worth mentioning that Wang \cite{2005Canonical} and Mossa-Loi \cite{loi-bundle} proved the uniqueness of balanced metrics on vector bundles on compact K\"{a}hler manifolds, which is related to algebraic stability of vector bundles. Of course, our method is very different from theirs due to the infinite-dimensional nature of our problem.

	\
	
	A corollary of our result is the continuous dependence of the solutions on $\beta$. 
	
	\begin{cor}\label{cor-main}
		Let $ ^\beta\blambda=(^\beta\lambda_i)$ be the unique solution to \eqref{e-2} with $ ^\beta\lambda_0=1$. Then for each $i$, $ ^\beta\lambda_i$ depends on continuously on $\beta$. Equivalently, the corresponding function $ ^\beta f(x)$ depends continuously on $\beta$ for $x$ in any finite interval.
	\end{cor}

	As our setting on $(\C,0)$ is expected to play the role of a model, it is then important to understand the asymptotics of $^\beta f(x)$ as $x\to \infty$. We make the following conjecture in this direction. 

\begin{conj}
	Let $^\beta f(x)$ be the unique solution to \eqref{e-2} with $^\beta (0)=1$, then $^\beta f$ should asymptotically be $C_\beta x^\beta e^x$ as $x\to +\infty$.
\end{conj}

	We will now talk about the proof.
	Let $u(x)=xf'(x)/f(x)$. In \cite{MilesWill}, the two main ingredients are a result of \cite{hayman1978} and the explicit formula for the coefficients of the Taylor expansion of $e^x$. More precisely, \cite{hayman1978} showed that when $\beta=0$, if $f(x)$ satisfies the equation \ref{e-2}, then $u-x=O(\sqrt{x})$ as $x\to \infty$. Although this estimate likely holds for $\beta>0$, we no longer have explicit formulas for the coefficients of $f(x)$. So we have to take a different route, in which we do not rely on the result from \cite{hayman1978}.

	Since the coefficients $e^{-\lambda_i}$ and the function $f$ are mutually determined, our proof involves improving our knowledge of both $e^{-\lambda_i}$ and $f$ simultaneously. To obtain good estimates of the $\lambda_i$'s, we need to control $\frac{du}{dx}=x\frac{d^2}{dt^2}\log f(x)$, where $t=\log x$. Noticing that $\frac{d^2}{dt^2}\log f(x)>0$ and $\lambda''(i)>0$, our main strategy is to make use of the convexity. In particular, in section 3, we will show that there is a specific type of convex function(called dominating function), which is computable and can be used to control the needed integrals of general convex functions. And that is the main observation and the most technical part of this article. Then, using this technique, we can show the asymptotic behavior of $\lambda_i$ as $i$ approaches infinity and the asymptotic behavior of $f(x)$ as $x$ approaches infinity, which then implies the uniqueness of $f(x)$.

	\

	This paper is organized as follows: In Section 2, we build up the rough estimates; In Section 3, we develop the tool of dominating function to control the integrals of general convex functions; Then in Section 4, we use the dominating function to prove more accurate estimates, which can then be used in section 5 to prove the uniqueness of the balanced metric.

	\
	
	\textbf{Acknowledgements.} The author would like to thank Professor Song Sun for many very helpful discussions.

	\section{rough estimates}
	We fix $\beta\in (0,1)$, and let $f(x)$ be an entire function satisfying the equation \ref{e-2}.	
	We make the normalization so that $f(0)=1$, i.e. $\lambda_0=0$.
	
	\subsection{degree 1 estimate}
Recall that we have defined $u=x\frac{f'(x)}{f(x)}$. 

\begin{prop}\label{prop-start}
	\begin{equation}
	\lim_{x\to \infty}\frac{u}{x}=1.
	\end{equation}
\end{prop}
The proof basically repeats the arguments in section 3.0.1 in \cite{sun-sun-2}, but for the convenience of the readers, we include it here.
\begin{proof}
	We use the arguments in section 3 in \cite{sun-sun-2}. First, we have
	$$\int_0^{\infty}\frac{f((1-t)x)}{f(x)}dx=\frac{1}{t}-\beta, $$
	and 
	$$\int_0^{\infty}\frac{f(x/(1+t))}{f(x)}dx=\frac{1+t}{t}-\beta. $$
	Then we get  for all $t>0$ that 
	$$|\int_0^{\infty} e^{-tu} \phi(u) du-\frac{1}{t}|\leq 1+\beta. $$
	So for any $n\in \N$, we have
	$$|\int_0^{\infty} e^{-tu}e^{-ntu} \phi(u) du-\frac{1}{(n+1)t}|\leq 1+\beta.$$
	Let $g$ be any measurable function on $[0,1]$ with bounded variation. Fix any $\epsilon>0$ small, by the Stone-Weierstrass approximation theorem there are polynomials $p$ and $P$ such that
	$$p(x)< g(x)< P(x), $$ 
	and
	$$\int_0^{\infty} e^{-u}(P(e^{-u})-p(e^{-u}))dt< \epsilon.$$
	For any $n\in \N$, we have
	$$|\int_0^{\infty} e^{-tu}e^{-ntu} \phi(u) du-\frac{1}{(n+1)t}|\leq C_0\frac{1+(n+1)^{1/2}t^{1/2}}{(n+1)^{1/2}t^{1/2}}.$$
	Let $N$ be the maximum of $\deg p$ and $\deg P$. So
	$$|\int_0^{\infty} e^{-tu} P(e^{-tu})\phi(u)du-\frac{1}{t} \int_0^{\infty} e^{-u} P(e^{-u}) du|\leq (1+\beta)(1+N). $$
	So 
	$$\int_0^{\infty} e^{-tu} g(e^{-tu})\phi(u) du\leq \frac{1}{t}(\int_0^{\infty} e^{-u} g(e^{-u})du+\epsilon)+(1+\beta)(1+N). $$
	Similarly, we have
	$$\int_0^{\infty} e^{-tu} g(e^{-tu})\phi(t) dt\geq \frac{1}{t}(\int_0^{\infty} e^{-u} g(e^{-u})du-\epsilon)-(1+\beta)(1+N). $$
	Now define $g(x)$ to be zero for $x< e^{-1}$ and $1/x$ for $x\in [e^{-1}, 1]$. Then 
	
	$$\int_0^{\infty} e^{-u} g(e^{-u}) du =1, $$
	and
	
	$$\int_0^{\infty} e^{-tu} g(e^{-tu})\phi(u)du=\int_0^{1/t} \phi(u) du. $$
	Let $T=1/t$ and fix $\epsilon>0$ sufficiently small we obtain 
	$$\int_0^T \phi(u)du\leq (1+\epsilon)T+(1+\beta)(1+N). $$
	Similarly, we have
	\begin{equation}
	\int_0^T \phi(u)du\geq (1-\epsilon)T-(1+\beta)(1+N). 
	\end{equation}
	So we get 
	$$1-\epsilon\leq \liminf_{x\to \infty}\frac{u}{x}\leq\limsup_{x\to \infty}\frac{u}{x}\leq 1+\epsilon.$$
	Since $\epsilon$ is arbitrary, the conclusion follows.

\end{proof}

\subsection{rough degree 2 estimates} 
	
	Denote by $t=\log x$.  Then we have  $$\int_0^\infty \frac{x^ndx}{f(x)}=\int_{-\infty}^\infty e^{(n+1)t-\log f(x)}dt,$$
	and
	$$f(x)=\sum_{i=0}^\infty e^{-\lambda_i+i\log x}.$$
	In both the integral and the summation, the exponents have concavity of which we will make heavy use.
	For any $a\in \mathbb R$, we denote $$g_a(t)=at-\log f(x).$$ It is easy to check that $g_a$ is a strictly concave function of $t$.  Since  $g_a'(t)=a-u(x)$, and by proposition \ref{prop-start} $\lim_{x\to \infty}\frac{f'(x)}{f(x)}=1$, we see that $g_a(t)$ has a unique maximum at $t_{a}=\log x_a$ with $u(x_{a})=a$.  Moreover, we have $\frac{dx_a}{da}>0$ and $\lim_{a\rightarrow\infty}a^{-1}x_a=1$. In particular, for $x$ large there is a unique $a$ such that $x=x_a$.

	 We also define the function 
	\begin{equation}c(a)=(\int_0^\infty \frac{x^a}{f(x)}dx)^{-1} \label{e:definition of ca}	
	\end{equation}
 and $\lambda(a)=-\log c(a)$ for $a\in [0, \infty)$. Notice that when $a>0$ is an integer we have $c(a)=c_a$ and $\lambda(a)=\lambda_a$. One also notices that $c(0)\neq c_0$, which can be ignored, since it will not affect our arguments.

 One can compute
	\begin{equation}\label{lambda derivative}\lambda'(a)=c(a)\int_0^\infty \frac{(\log x)x^adx}{f(x)},	
	\end{equation}
 and
	\begin{eqnarray}\label{lambda second derivative}
		\frac{d^2}{da^2}\lambda(a)=c^2(a)[ \int_0^\infty \frac{x^adx}{f(x)}\int_0^\infty \frac{(\log x)^2x^adx}{f(x)}-( \int_0^\infty \frac{(\log x)x^adx}{f(x)})^2] >0.
	\end{eqnarray}
	For $x>1$ we denote $$F_x(n)=n\log x-\lambda(n),$$ then  $F_x(n)$ is a strictly concave function of $n$.   
	It follows that $F_x(n)$ has at most one maximum. Estimate (3.9) in \cite{sun-sun-2} implies that $\lambda'(n)\to +\infty$ as $n\to+\infty$. Therefore when $x$ is large $F_x(n)$ attains its maximum at $n_x>1$. Moreover, as $x\rightarrow\infty$, we have $n_x\rightarrow\infty$.

	Now for all $a\geq1$ we denote $$h_a(x)=e^{-g_a(x)}=\frac{f(x)}{c(a)x^a}.$$

	\begin{lem}\label{lem13}
		For $a\gg1$, we have $h_a(x_a)\geq \frac{\sqrt{a}}{13}$.
	\end{lem}
	\begin{proof}
		For simplicity, we assume $a$ is an integer. The case when $a$ is not an integer is similar, just with more complicated notations.
		We write  $$h_a(x)=\sum_{i=0}^\infty \delta_i(x),$$ where $\delta_i(x)=c_a^{-1}c_ix^{i-a}$. Then for $a\geq 1$, we have $\int_0^\infty h_a(x)^{-1}dx=1$.
		
		Assume that $h_a(x_a)<\frac{\sqrt{a}}{13}$ for some $a\gg1$. Let $\lambda_x(i)=F_x'(i)=\log x +\frac{c'(i)}{c(i)}$. By definition, $\lambda_x(n_x)=0$ and $\delta_a(x_a)=1$ . Let $\gamma=\delta_{a+\asr}(x_a)$.  Then we have, by concavity,
		$$\sum_{i=a}^{a+\asr}\delta_i(x_a)\geq \asr \min (\gamma,1).$$
		It follows that $\gamma<1$  and \begin{equation}\label{bound on n}
 	n_{x_a}<a+\sqrt{a}.
 \end{equation}
 Then we have $\gamma<\frac{\sqrt{a}}{13\asr}$. Furthermore we must have  $\lambda_{x_a}(a+\asr)<0$.

		By the mean value theorem, $-\lambda_{x_a}(a+\asr)\asr>\log \frac{1}{\gamma}$. Therefore $$\lambda_{x_a}(a+\asr)<\frac{-1}{\asr}\log(\frac{13\asr}{\sqrt{a}}).$$
		Then \begin{eqnarray*}
			\lambda_{x_{a}+\sqrt{a}}(a+\asr)&=&\lambda_{x_a}(a+\asr)+\log(1+\frac{\sqrt{a}}{x_a})\\&<&\frac{\sqrt{a}}{x_a}-\frac{1}{\asr}\log(\frac{13\asr}{\sqrt{a}})\\&<&0,
		\end{eqnarray*}
		if $a$ is sufficiently large. Therefore, $n_{x_a+\sqrt{a}}<a+\asr.$
		So for $0<y<\sqrt{a}$, we  have
		\begin{equation}
		\sum_{i>a+\sqrt{a}}\delta_i(x_a+y)<\frac{\gamma(1+\frac{y}{x_a})^{\asr}}{-(\frac{\sqrt{a}}{x_a}-\frac{1}{\asr}\log(\frac{13\asr}{\sqrt{a}}))}<\frac{e^2\sqrt{a}}{26},
		\end{equation}
		where the second inequality is because the denominator is larger than $\frac{2}{\sqrt{a}}$ and $(1+\frac{y}{x_a})^{\asr}<e^2$ for $a$ large.
		
		Since $\frac{\delta_i(x_a+y)}{\delta_i(x_a)}=(1+\frac{y}{x_a})^{i-a}$, we also have
		\begin{equation}
		\sum_{a\leq i\leq a+\sqrt{a}}\delta_i(x_a+y)<(1+\frac{y}{x_a})^{\asr}\sum_{a\leq i\leq a+\sqrt{a}}\delta_i(x_a)<\frac{e^2\sqrt{a}}{13}.
		\end{equation}
		So we have $h_a(x_a+y)<(1+e^2+\frac{e^2}{2})\frac{\sqrt{a}}{13}<\sqrt{a}$. Therefore
		$$\int_0^{\infty}\frac{1}{h_a(x)}dx>1,$$
		which is a contradiction.
	\end{proof}

	We can write $$h_a(x)=\frac{1}{c_a}\sum_{i=0}^\infty e^{(i-a)t+\log c(i)}.$$
	Then
	$$\frac{d}{dt}\log h_a(x)=\frac{\sum_{i=0}^\infty (i-a)e^{(i-a)t+\log c(i)}}{\sum_{i=0}^\infty e^{(i-a)t+\log c(i)}},$$and 
	\begin{eqnarray*}
		\frac{d^2}{dt^2}\log h_a(x)&=&\frac{h_a\frac{d^2}{dt^2}h_a-(\frac{d}{dt}h_a)^2}{h^2_a}\\
		&=&\sum_{i=0}^\infty (i-a-\sigma_a(x))^2\tau_x(i),
	\end{eqnarray*}
	where $\tau_x(i)=h_i(x)^{-1}$, and $\sigma_a(x)=\frac{\frac{d}{dt}h_a(x)}{h_a(x)}$. 
	
	We denote $\tilde t_a=\log \tilde x_a=\frac{d}{da}\lambda(a)$. Then $n_{\tilde x_a}=a$ 
	and $h_a(\tilde x_a)\geq h_a(x_a)\geq \frac{\sqrt{a}}{13}.$  The convexity of $\lambda(n)$ implies that $\tilde x_a$ is a strictly increasing function of $a$. In particular, for $x$ large, there is a unique $a$ such that $x=\tilde x_a$.
	\begin{lem}
		Suppose there exists  $A>0$ such that $h_a(x_a)>\frac{\sqrt{a}}{A}$ for all $a$ large enough. Then for any $\delta>0$, we have
		$$\frac{d^2}{dt^2}\log f(x)>\frac{x}{27A^2+\delta},$$
		for $x$ large enough.
	\end{lem}
	\begin{proof}
		Fix $\delta>0$. For $x$ large we may write  $x=\tilde x_a$ for a unique large $a$. By concavity of $F_x$ 
		we have $\tau_x(i)\leq \frac{1}{h_a(x)}$ for all $i$. So if we take the interval $I=(\sigma_a(x)+a-\frac{h_a(x)}{3},\sigma_a(x)+a-\frac{h_a(x)}{3})$,
		then $\sum_{i\in I}\tau_x(i)<\frac{2}{3}$. So $\sum_{i\notin I}\tau_x(i)>\frac{1}{3}$. Therefore 
		$$\sum_{i=0}^\infty (i-a-\sigma_a(x))^2\tau_x(i)>\frac{1}{3}(\frac{h_a(x)}{3})^2>\frac{a}{27A^2}$$
		If $\tilde x_a<x_a$, then since $\lim_{a\to\infty}\frac{a}{x_a}=1$, we get that for $a$ large that
		$$\frac{d^2}{dt^2}\log h_a(\tilde x_a)>\frac{\tilde x_a}{27A^2+\delta}.$$
		If $\tilde x_a>x_a$, then $n_{x_a}<a$. So 
		$\lambda_{x_a}(a)=\log \frac{x_a}{\tilde x_a}<0$. Hence
		 $$c_a^{-1}\sum_{i\geq a+1}(i-a)c_ix_a^{i-a}<\sum_{j\geq 1}je^{j\log\frac{x_a}{\tilde x_a}}=\frac{x_a}{\tilde x_a(1-x_a/\tilde x_a)^2}.$$
		Since $\sum_i (i-a)c_ix_a^{i-a}=0$, we have 
		$$c_a^{-1}\sum_{i< a}(a-i)c_ix_a^{i-a}<\frac{x_a}{\tilde x_a(1-x_a/\tilde x_a)^2}+1$$
		On the other hand, $$c_a^{-1}\sum_{n_{x_a}\leq i< a}(a-i)c_ix_a^{i-a}\geq \sum_{j=1}^{[a]-\ulcorner n_{x_a}\urcorner}j\geq \frac{1}{2}([a]-\ulcorner n_{x_a}\urcorner)^2.$$
		We denote  $r_a=[a]-\ulcorner n_{x_a}\urcorner$  and $\beta_a=\frac{x_a}{\tilde x_a}$. So $$ \frac12 r_a^2<\frac{\beta_a}{(1-\beta_a)^2}+1.$$
		Let $\xi\in (n_{x_a},a)$ such that $\frac{c'(\xi)}{c(\xi)}-\frac{c'(a)}{c(a)}=\frac12 \log \frac{\tilde x_a}{x_a}$.
		Then we can estimate $$h_\xi(\tilde x_\xi)<r_a+4+\frac{2\sqrt{\beta_a}}{1-\sqrt{\beta_a}}.$$ We have $h_\xi(\tilde x_\xi)\geq\frac{\sqrt{\xi}}{13}$. When $a\to \infty$, we have $x_a\to \infty$, so $n_{x_a}\to \infty$. In particular, when $a$ is large, $\xi$ is also large. We have $$\beta_a\geq \frac{27A^2x_a}{(27A^2+\delta)a}.$$
		So we always have $$\frac{d^2}{dt^2}\log f(x)=\frac{d^2}{dt^2}\log h_a(\tilde x_a)>\frac{\tilde x_a}{27A^2+\delta},$$
		
	\end{proof}
	Similarly, we get
	\begin{lem}
		Suppose there exists $A>0$ such that $h_a(x_a)>\frac{\sqrt{a}}{A}$ for all $a$ large enough. 
		Then for any $\delta>0$, we have
		$$\frac{d^2}{da^2}\lambda(a)>\frac{1}{(27A^2+\delta) a},$$
		for $a$ large enough.
		
	\end{lem}
	\begin{proof}
		From \eqref{lambda derivative} we have  $$\tilde t_a=\int_0^\infty\frac{\log x dx}{h_a(x)}.$$ So  by \eqref{lambda second derivative} we get
		$$\frac{d^2}{da^2}\lambda(a)=\int_0^\infty\frac{(\log x-\tilde t_a)^2dx}{h_a(x)}.$$
		 We take $I=(\tilde x_a-\frac{\sqrt{a}}{3A},\tilde x_a+\frac{\sqrt{a}}{3A})$, then since $\frac{1}{h_a(x)}\leq\frac{1}{h_a(x_a)}<\frac{A}{\sqrt{a}}$, we have  $\int_{x\in I}\frac{dx}{h_a(x)}< \frac{2}{3}$. So $\int_{x\notin I}\frac{dx}{h_a(x)}> \frac{1}{3}$. 
		Now $$\frac{d^2}{da^2}\lambda(a)>\int_{x\notin I}\frac{(\log x-\tilde t_a)^2dx}{h_a(x)}>\frac{1}{3}\log^2(1+\frac{\sqrt{a}}{3A \tilde x_a}).
		$$
		As we have shown in the last lemma for any $\epsilon>0$, we have $\frac{\tilde x_a}{a}>1-\epsilon$ for $a$ large enough. Therefore, we have
		$$\frac{d^2}{da^2}\lambda(a)>\frac{1}{(27A^2+\delta)a},$$
		for $a$ large enough.
	\end{proof}
	Now let $A=13$, the above results imply the following.
	\begin{lem}\label{cor-13}
	For $x\gg1$ we have $$\frac{d^2}{dt^2}\log f(x)>\frac{x}{5000},$$ and for $a\gg1$ we have
	$$\frac{d^2}{da^2}\lambda(a)>\frac{1}{5000a}.$$ 	
	\end{lem}

	As corollaries, we have the following.
	
	\begin{prop}\label{prop-tatilde}
		For $a$ large we have $\tilde t_a=t_{a+1}+O(\frac{\log a}{\sqrt{a}})$
	\end{prop}
\begin{proof}
Notice $t-\log h_a(x)=g_{a+1}(t)+\log c(a)$, and $g_{a+1}(t)$ has unique maximum at $t_{a+1}$.  By Lemma \ref{cor-13} and the standard Laplace's method we see that there  exists a constant $C>0$ so that $$\int_{|t-t_{a+1}|>C\frac{\log a}{\sqrt{a}}}\frac{t}{h_a(x)}dx=\epsilon(a),$$
where $\epsilon(a)$ denotes a quantity which is $O(a^{-k})$ as $a\rightarrow\infty$ for all $k>0$. 
	So by \eqref{lambda derivative} we have $$\tilde t_a=t_{a+1}+\int_{|t-t_{a+1}|\leq C\frac{\log a}{\sqrt{a}}}\frac{t-t_{a+1}}{h_a(x)}dx+\epsilon(a)=t_{a+1}+O(\frac{\log a}{\sqrt{a}}).$$
\end{proof}
	
	\begin{prop}\label{prop-big-upp}
		We have $$\frac{d^2}{da^2}\lambda(a)<2*10^7 \frac{1}{a}$$ for $a$ large enough, and $$\frac{d^2}{dt^2}\log f(x)<2*10^7 x$$ for $x$ 
		large enough.
		
	\end{prop}
	\begin{proof}
		Let $x=\tilde x_a$, then $\tau_x(i)$ achieves its maximum at $i=a$. So
		\begin{eqnarray*}
			\frac{d^2}{dt^2}\log f(x)=\frac{d^2}{dt^2}\log h_a(x)
			&=&\sum (i-a-\frac{d}{dt}h_a(x))^2\tau_x(i)\\
			&\leq &\sum (i-a)^2\tau_x(i)\\
			&\leq &\frac{1}{h_a(\tilde x_a)}(O(1)+\int_{-\infty}^\infty y^2e^{-\frac{y^2}{10000a}}dy)\\
			&\leq &\frac{13}{\sqrt{a}}(O(1)+\frac{(10000a)^{3/2}\sqrt{\pi}}{2})\\
			&\leq &1.2*10^7 a\\
			&\leq &2*10^7 \tilde x_a,
		\end{eqnarray*}
		
		Therefore $\frac{du}{dx}<2*10^7$. So $x_{a+1}=x_a+O(1)$. 
	Similarly, since $$\frac{d^2}{da^2}\lambda(a)=\int\frac{(\log x-\tilde t_a)^2dx}{h_a(x)}
		\leq \int\frac{(\log x-t_{a+1})^2dx}{h_a(x)},$$ we can substitute $y=\log x-t_{a+1}$ to get
		\begin{eqnarray*}
			\frac{d^2}{da^2}\lambda(a)
			&\leq &\frac{x_{a+1}}{h_a(x_{a+1})}\int_{-\infty}^{\infty}e^{-\frac{a}{11000}}y^2dy\\
			&\leq &14\sqrt{a}\frac{\sqrt{\pi}(11000)^{3/2}}{2a^{3/2}}\\
			&\leq &2*10^7 \frac{1}{a}.\\
		\end{eqnarray*}
		
	\end{proof}
Now we get a more quantitative estimate on $x_a$.

	\begin{prop}
		$x_a=a+O(\sqrt{a}\log a)$  for $a$ large enough.
	\end{prop}
	\begin{proof}
		First, we have $$\int_0^{x_a+\sqrt{a}\log a}\frac{c_ix^i}{f(x)}dx=1-\beta\delta_{0i}+\epsilon(a)$$ for $0\leq i\leq a$. So
		$$x_a+\sqrt{a}\log a>\int_0^{x_a+\sqrt{a}\log a}\frac{\sum_{i=0}^ac_ix^i}{f(x)}dx=a+1-\beta+\epsilon(a)$$
		For the other direction, we have $$\frac{\sum_{i>a}c_ix^i}{f(x)}=\epsilon(a)$$ for $x\leq x_a-\sqrt{a}\log a$.
		Therefore, 
		$$a+1>\int_0^{x_a-\sqrt{a}\log a}\frac{\sum_{i=0}^ac_ix^i}{f(x)}dx=x_a-\sqrt{a}\log a+\epsilon(a).$$
		
	\end{proof}

	\begin{remark}
	With these estimates, one can improve the estimates in Lemma \ref{lem13}
	to get $A=2.5$, then we have that $\frac{d^2}{dt^2}\log f(x)>\frac{x}{150}$ for $x$ large and $\frac{d^2}{da^2}\lambda(a)>\frac{1}{150a}$ for $a$ large. 
	Then one can generate better upper bounds by repeating the proof of Proposition \ref{prop-big-upp}. But this refinement will not be essential to us.
	To further narrow down the gap between the upper and lower bounds, we need the techniques in the following subsection.
	\end{remark}

\section{Dominating function}
In this section, we will develop the tool of dominating functions to narrow down the estimates for $\frac{d^2}{da^2}\lambda(a)$ and $\frac{d^2}{dt^2}\log f(x)$.
\subsection{half real line}
	
	\begin{theorem}\label{convex-general}
		Let $g(t)$ be a $C^{1}$ function on $\R_+$ satisfying:
		\begin{itemize}
			\item[(1)]$g(0)=0$;
			\item[(2)]$g'(0)=0$;
			\item[(3)]The left derivative and right derivative of $g'(t)$ exist for every $t$; 
			\item[(4)]$g(t)\in C^{2}$ except at finite points;
			\item[(5)]there exist $0<M_1<M_2<\infty $ such that $M_1<g''(t)<M_2$ for all  $t\in \R $.  	
		\end{itemize}
		Denote $$a(g)=\int_{0}^\infty e^{-g(t)}dt,$$ $$b(g)=\int_{0}^\infty t^2 e^{-g(t)}dt,$$ and $$d(g)=\frac{b(g)}{4(a(g))^3}.$$
		Then there exist $\frac{1}{12}<q<p<\frac12$ depending only on $M_1$ and $M_2$ such that
		$$\frac{1}{12}<q<d(g)<p<\frac12.$$
	\end{theorem}
	\begin{remark}
		Notice that if $\tilde g(y)=g(\rho y)$ for some $\rho>0$,
	then $d(\tilde g)=d(g)$. So the bounds $p$ and $q$  depend only on the ratio $\frac{M_2}{M_1}$.
	\end{remark}
	
	\begin{proof}

		For each pair $0<s<l<\infty$, we can define a function $g_{s,l}(t)$ satisfying $g_{s, l}(0)=g_{s,l}'(0)=0$ and 
		$$g_{s,l}''(t)=\begin{cases}
			M_2, &t< s\\
			g''(t), &l>t\geq s\\
			M_1, &t\geq l
		\end{cases} $$
		We call a pair $0<s<l<\infty$ an admissible pair if $a(g_{s,l})=1$. It is easy to see that when $s$ is small enough, there exists  $l>s$ such that $(s,l)$ is admissible.
		Notice that for an admissible pair $(s,l)$ the following holds:
		$$\exists\ l_1>l \ \  \textit{satisfying}\begin{cases}
			g(t)\leq g_{s,l}(t), & t\leq l_1\\
			g(t)\geq g_{s,l}(t), & t>l_1
		\end{cases}   $$ 
		Therefore 
		$$b(g)-b(g_{s,l})=\int_{0}^\infty (t^2-l_1^2)(e^{-g(t)}-e^{-g_{s, l}(t)})dt\leq0.$$
		Hence, we have $d(g)\leq d(g_{s,l}).$ 

		\textbf{For later references, we will call this operation that replaces $g$ with $g_{s,l}$ operation A. }

		Let $l(s)$ be the smallest $l$ such that $(s,l)$ is admissible. Then by a straightforward continuity argument, one sees that $\exists s_0>0$ such that $s_0=l(s_0)$. Clearly, $g_{s_0,l(s_0)}$ is determined by the pair $(M_1, M_2)$, so we can denote by $d_{M_1,M_2}=d(g_{s_0,l(s_0)})$. Then one easily sees that for $0<M_1'\leq M_1\leq M_2\leq M_2'$, we have $d(g)\leq d_{M_1',M_2'}$. So for $\delta>0$ small enough, we have 
		$$d(g)\leq d_{\delta,1/\delta}.$$
		We denote by $s_\delta$ the $s_0$ corresponding to the pair $(\delta,1/\delta)$, and by $g_\delta$ the corresponding function $g_{s_0,l(s_0)}$. Then we have $$g_\delta(t)=\begin{cases}
		\frac{t^2}{2\delta}, & t\leq s_\delta\\
		\frac{s_\delta^2}{2\delta}+\frac{s_\delta}{\delta}|t-s_\delta|+\frac{\delta}{2}(t-s_\delta)^2, & t> s_\delta
		\end{cases}$$
		Then in order for $a(g_\delta)=1$, $s_\delta$ has to converge to $0$ when $\delta\to 0$ and $\lim_{\delta\to 0}\frac{s_\delta}{\delta}=1$. So $$\lim_{\delta\to 0}g_\delta(t)=g_0(t)=t.$$
		Then by the dominated convergence $\lim_{\delta\to 0}d(g_\delta)=d(g_0)=\frac{1}{2}$. Therefore we have proved the upper bound.

		For the lower bound, the argument is similar. We also replace $g$ with $\tilde{g}_{s,l}$ satisfying $$\tilde{g}_{s,l}''(t)=\begin{cases}
			M_1, &t< s\\
			g''(t), &l>t\geq s\\
			M_2, &t\geq l
		\end{cases} $$

		\textbf{For later references, we will call this operation that replaces $g$ with $\tilde{g}_{s,l}$ operation B. }

		And we also get a function 
		$$g_\delta(t)=\begin{cases}
			\frac{\delta}{2}t^2, & t\leq s_\delta\\
			\frac{\delta}{2}s_\delta^2+\delta s_\delta(t-s_\delta)+\frac{1}{2\delta}(t-s_\delta)^2, & t> s_\delta
			\end{cases}$$
		satisfying $d(g)>d(g_\delta)$. Again in order for $a(g_\delta)=1$, $s(\delta)$ has to converge to $1/2$ when $\delta\to 0$ and $$\lim_{\delta\to 0}g_\delta(t)=g_0(t)=\chi_I(t),$$
		where $\chi_I(t)$ is the characteristic function of the interval $I=[0,1]$.
Moreover,  $a(g_0)=1, b(g_0)=\frac{1}{3}$, so $d(g_0)=\frac{1}{12}$. 
		Again by dominated convergence $\lim_{\delta\to 0}d(g_\delta)=d(g_0)=\frac{1}{12}$. This proves the lower bound.
	\end{proof}

	By dilation, in the proof of Theorem \ref{convex-general} we have actually proved the following two lemmas:

	\begin{lem}\label{lem-surgery}
		Let $g(y)$ be a  function satisfying the conditions in Theorem \ref{convex-general}. Let $m=\frac{M_1}{M_2}$.
		Then there exist $c>0$ and a function $f(y)$ satisfying:
		\begin{itemize}
			\item $f(y)=\frac{y^2}{2}$ for $y\leq c$;
			\item $f(y)=\frac{my^2}{2}+(1-m)cy+\frac{mc^2}{2}-\frac{c^2}{2}$ for $y>c$;
			\item $d(f)\geq d(g)$.
		\end{itemize}
	\end{lem}
	\begin{lem}\label{lem-surgery2}
	Let $g(y)$ be a convex function satisfying the conditions in Theorem \ref{convex-general}. Let $m=\frac{M_2}{M_1}$.
	Then there exist $c>0$ and a function $f(y)$ satisfying:
	\begin{itemize}
		\item $f(y)=\frac{y^2}{2}$ for $y\leq c$;
		\item $f(y)=\frac{my^2}{2}+(1-m)cy+\frac{mc^2}{2}-\frac{c^2}{2}$ for $y>c$.
		\item $d(f)\leq d(g)$.
	\end{itemize}
\end{lem}

	Now for fixed $m>0$, we consider the function 
	$$d_m(c)=\frac{\int_0^\infty y^2e^{-g_c(y)}dy}{(\int_0^\infty e^{-g_c(y)}dy)^3},$$
	where $g_c(y)$ satisfies:
	\begin{itemize}
		\item $g_c(y)=\frac{y^2}{2}$ for $y\leq c$;
		\item $g_c(y)=\frac{my^2}{2}+(1-m)cy+\frac{mc^2}{2}-\frac{c^2}{2}$ for $y>c$.
	\end{itemize} 
	
	 And clearly $d_m(0)=\frac{2}{\pi}$ and $\lim_{c\to\infty}d_m(c)=\frac{2}{\pi}$.
	So both the maximum and minimum of $d_m(c)$ in $[0,\infty]$ exist. 
	\begin{remark}
		Later on, when we need to emphasize the fact that $g_c(y)$ also depends on $m$, we will denote it by $g_{m,c}(y)$
	\end{remark}

	Let $$a(c)=\int_0^\infty e^{-g_c(y)}dy$$ and $$b(c)=\int_0^\infty y^2e^{-g_c(y)}dy.$$ Then since
	$$\frac{\partial}{\partial c}g_c(y)=(1-m)(y-c),$$ we have
	$$a'(c)=e^{\frac{1}{2}c^2-\frac{1}{2}mc^2}\int_c^{\infty}(m-1)(x-c)e^{-\frac{1}{2}mx^2+(m-1)cx}dx$$
	and $$b'(c)=e^{\frac{1}{2}c^2-\frac{1}{2}mc^2}\int_c^{\infty}(m-1)(x-c)x^2e^{-\frac{1}{2}mx^2+(m-1)cx}dx.$$
	Denote  $$\gamma(c)=a(c)b'(c)-3b(c)a'(c),$$ then 
 $d'(c)=\frac{\gamma(c)}{a(c)^4}$, and

	\begin{eqnarray*}
		\gamma(0)&=&(m-1)[\int_0^\infty e^{-\frac{1}{2}mx^2}dx  \int_0^\infty x^3e^{-\frac{1}{2}mx^2}dx\\ &&-3\int_0^\infty x^2e^{-\frac{1}{2}mx^2}dx\int_0^\infty xe^{-\frac{1}{2}mx^2}dx]\\
		&=&(m-1)[\sqrt{\frac{\pi}{2}}\frac{1}{\sqrt{m}}\cdot 2\frac{1}{m^2}-3\sqrt{\frac{\pi}{2}}\frac{1}{m^{3/2}}\cdot \frac{1}{m} ]\\
		&=&(1-m)\sqrt{\frac{\pi}{2}}m^{-5/2}
	\end{eqnarray*}
	So $d'(0)>1$ when $m<1$, and $d'(0)<1$ when $m>1$. 
	
	By Theorem \ref{convex-general}, we have  $\frac{b(c)}{a^3(c)}<\frac12\cdot 4=2$. Since $$a(c)\leq \sqrt{\frac{\pi}{2}} \max \{1,\sqrt{\frac{1}{m}} \},$$ 
	we get that $$\frac{b(c)}{a(c)}<\pi\max \{1,\frac{1}{m} \}.$$ On the other
	hand clearly we have $|\frac{b'(c)}{a'(c)}|\geq c^2$. Therefore, we have 
	\begin{itemize}
		\item when $m<1$, $d'(c)<0$ for $c$ large enough;
		\item when $m>1$, $d'(c)>0$ for $c$ large enough.
	\end{itemize}
	So when $m<1$, $d_m(c)$ attains its maximum for some $c>0$, and when $m>1$, $d_m(c)$ attains its minimum for some $c>0$.
	For $m\geq 1$, we denote $$p(m)=\frac{1}{4}\max_{c\geq 0}d_{1/m}(c),$$ and $$q(m)=\frac{1}{4}\min_{c\geq 0}d_{m}(c),$$ $$F(m)=(\frac{p(m)}{q(m)})^2.$$ Clearly,  $\lim_{m\to 1}F(m)=1$.  The rest of this subsection is devoted to the proof of the following quantitative estimate
	\begin{prop}\label{p:F derivative estimate}
		For $m\in [1, 1.01)$ we have 
	$$0<F'(m)<0.84.$$
	\end{prop}
Let $$G(m,c)=a(c)\int_c^\infty (x-c)x^2e^{-g_c(x)}dx-3b(c)\int_c^\infty (x-c)e^{-g_c(x)}dx,$$ then 
	$G(1,c)=\sqrt{\frac{\pi}{2}}\int_c^\infty (x-c)(x^2-3)e^{-x^2/2}dx$. It is not hard to see that 
	$G(1,c)=0$ has exactly one solution $c=c_0$ for $c\geq 0$, and we can use Mathematica to estimate
	$c_0\approx 0.612003$. 
	Also, we have:
	\begin{eqnarray*}
		\frac{\partial}{\partial c}G&=&3b(c)\int_c^\infty e^{-g_c}dx-a(c)\int_c^\infty x^2e^{-g_c}dx\\
		&+&(1-m)[\int_c^\infty (x-c)e^{-g_c}dx+a(c)\int_c^\infty (x-c)^2x^2 e^{-g_c}dx\\
		&-&3\int_c^\infty (x-c)x^2 e^{-g_c}dx\int_c^\infty (x-c) e^{-g_c}dx-3b(c)\int_c^\infty(x-c)^2 e^{-g_c}dx)]\\
		&=&A_1(c)+(1-m)A_2(c).
	\end{eqnarray*}
	Of course, $A_1(c)$ and $A_2(c)$ also depend on $m$. 
	Then $$\frac{\partial}{\partial c}G(1,c_0)\approx 1.06 $$
	So by the implicit function theorem, the equation $G(m,c)=0$ has exactly one solution $c(m)$ close to $c_0$ for $m$ in a neighborhood of $1$.
	Also, since
	$$\frac{\partial}{\partial m}g_c(y)=\frac12(y-c)^2,$$
	we have
	$$\frac{\partial}{\partial m}a(c)=-\frac12 \int_c^\infty (x-c)^2 e^{-g_c}dx,$$
	and $$\frac{\partial}{\partial m}b(c)=-\frac12 \int_c^\infty (x-c)^2x^2 e^{-g_c}dx.$$
	Therefore,
	\begin{eqnarray*}
		\frac{\partial}{\partial m}G&=&-\frac12 \int_c^\infty (x-c)^2 e^{-g_c}dx\int_c^\infty (x-c)x^2 e^{-g_c}dx
		-\frac12 a(c)\int_c^\infty  (x-c)^3e^{-g_c}dx\\
		\\ &+&\frac32 \int_c^\infty (x-c)^2x^2e^{-g_c}dx\int_c^\infty (x-c) e^{-g_c}dx+\frac32 b(c)\int_c^\infty(x-c)^3 e^{-g_c}dx
	\end{eqnarray*}
	We can use Mathematica again to estimate $\frac{\partial}{\partial m}G(1,c_0)\approx 1.557$. 
	So $\frac{d}{dm}c(1)\approx 1.47$.
	
	When $1\leq m<1.01$ and $c_0-0.03\leq c\leq c_0+0.03$,
	we have $$\frac{\partial}{\partial m}\int_c^\infty e^{-g_c}dx<0,$$ $$\frac{\partial}{\partial m}\int_c^\infty x^2e^{-g_c}dx<0,$$
	$$\frac{\partial}{\partial c}\int_c^\infty e^{-g_c}dx<0,$$ and $$\frac{\partial}{\partial c}\int_c^\infty x^2e^{-g_c}dx<0.$$ Therefore, we can estimate 
	$$0.932<A_1(c)<1.182.$$
	Similarly, we can estimate $|A_2(c)|<3.36$. Hence we have 
	$$0.89<\frac{\partial}{\partial c}G(m,c)<1.22,  $$
	for $1\leq m<1.01$ and $c_0-0.03\leq c\leq c_0+0.03$.
	In the same way, we can estimate 
	$$1.29<\frac{\partial}{\partial m}G(m,c)<1.792  $$
	So for $1\leq m<1.01$ the equation $G(m,c)=0$ has exactly one solution $c(m)$ satisfying $$1.05<c'(m)<2.02.$$
	
	When $0.99< m\leq 1$, we have similar conclusions. More precisely, we have
	$0.943<A_1(c)<1.197$ and
	$|A_2(c)|<4.21$. So we have 
	$$0.9<\frac{\partial}{\partial c}G(m,c)<1.24.  $$
	Also,
	$$1.32<\frac{\partial}{\partial m}G(m,c)<1.84  .$$
	Therefore 
	$$1.06<c'(m)<2.05 $$

	Now the argument above, we have for $1\leq m<1.01$,
	$q(m)=\frac{1}{4}d_m(c(m))$ and $p(m)=\frac{1}{4}d_{1/m}(c(\frac{1}{m}))$. So 
	$$q'(m)=\frac{1}{4(a(c(m)))^4}[\frac{d}{dm}b(c(m))a(c(m))-3\frac{d}{dm}a(c(m))b(c(m))].$$
	Let $\eta= \frac{d}{dm}b(c(m))a(c(m))-3\frac{d}{dm}a(c(m))b(c(m))$.
	We have
	\begin{eqnarray*}
		-\eta&=&\int_c^\infty \frac{1}{2}(x-c)^2x^2 e^{-g_c}dx a(c(m))-c'(m)b'(c)a(c)\\
		&&-3b(c)\int_c^\infty \frac12(x-c)^2 e^{-g_c}dx+3b(c)c'(m)a'(c)\\
		&=&\frac{1}{2}\int_0^\infty  e^{-g_c}dx\int_c^\infty (x-c)^2x^2 e^{-g_c}dx\\ &&-\frac{3}{2}\int_0^\infty x^2 e^{-g_c}dx\int_c^\infty (x-c)^2 e^{-g_c}dx.
	\end{eqnarray*}
	Then $\eta(1)\approx -0.318018$. 
	Since $\lim_{m\to 1}a(c(m))=\sqrt{\frac{\pi}{2}}$ and  $$\lim_{m\to 1}q(m)=\lim_{m\to 1}p(m)=\frac{1}{2\pi},$$ we get 
	$$F'(1)\approx 0.81 $$
	
	If $1\leq m<1.01$ and $|c-c_0|<0.021$, we can estimate $|\frac{d}{dm}\eta(c(m))|<4$, so we have $-0.314>\eta>-0.323$.
	Since $\frac{d}{dm}a(c(m))<0$, we have $1.251<a(c)<\sqrt{\frac{\pi }{2}}$, so 
	$$-0.132<4q'(m)<-0.127. $$
	We can also estimate $4q(m)>0.635$ and $4p(m)<0.638$, so 
	$$\frac{p(m)}{q(m)}<1.005$$

	Similarly if $0.99<m<1$ and $|c-c_0|<0.021$, since
	$\sqrt{\frac{\pi }{2}}<a(c)<1.2554$, we have
	$$-0.131<4p'(1/m)<-0.126.$$
		Therefore, 
	we have when $m\in [1, 1.01]$, 
	$$0<F'(m)<0.84.$$ This finishes the proof of Proposition \ref{p:F derivative estimate}.
	
	\subsection{full real line}	
	\begin{definition}
		We say that a function $g(t)$ is a $(M_1,M_2)$-admissible function if $g(t)$ is a convex function on $\R$ such that $g(t)$ and $g(-t)$, $t\geq 0$, satisfy the conditions in theorem \ref{convex-general} with constants $M_1,M_2$.
	\end{definition}
	  Let $g(t)$ be a $(M_1,M_2)$-admissible function, in order to control $\frac{d^2}{da^2}\lambda(a)$ and $\frac{d^2}{dt^2}\log f(x)$, we actually need to consider the following functions.
	Let $$\bar{t}(g)=\frac{\int_{-\infty}^{\infty }te^{-g(t)}dt}{\int_{-\infty}^{\infty }e^{-g(t)}dt},$$
	be the center of mass. We define the following functions of $g$: $$\tilde{a}(g)=\int_{-\infty}^{\infty }e^{-g(t)}dt,$$
	$$\tilde{b}(g)=\int_{-\infty}^{\infty }(t-\bar{t})^2e^{-g(t)}dt,$$
	and 
	$$\tilde{d}(g)=\frac{\tilde{b}(g)}{(\tilde{a}(g))^3}. $$
	\begin{remark}
		Notice that $\tilde{d}(g)$ does not change if we dilate $t$, namely $\tilde{d}(g(t))=\tilde{d}(g(\lambda t))$. 
	\end{remark}
	
	When $\tilde{a}(g)=1$, we have a simpler formula $$\tilde{d}(g)=\tilde{b}(g)=\int_{-\infty}^{\infty }t^2e^{-g(t)}dt-\bar{t}^2(g).$$

	Recall we have defined the function $g_{m,c}(y)$ right after lemma \ref{lem-surgery2}. 
	\begin{definition}
		For fixed $c_1\geq 0, c_2\geq 0$, we denote by $g_{m,c_1,c_2}(y)$ the function satisfying the following:
		\begin{itemize}
			\item $g_{m,c_1,c_2}(y)=g_{m,c_1}(y)$ for $y\geq 0$;
			\item $g_{m,c_1,c_2}(y)=g_{m,c_2}(-y)$ for $y<0$.
		\end{itemize}
	\end{definition}
We will need to apply operations that replace $g_{m,c_1,c_2}$ with $g_{m,c_1-\epsilon,c_2+\delta}$. We will make sure that the total mass $\tilde{a}(g)$ does not change. Then in order to show that we change $\frac{\int_{-\infty }^{\infty}t^2e^{-g}dt}{(\int_{-\infty }^{\infty}e^{-g}dt)^3}$ in a direction we need, we only
need to consider the ratio of the differences of the quantities $a(g)$ and $b(g)$ when we replace $g_{m,c}$ with $g_{m,c+\epsilon}$. And we can consider it from the infinitesimal point of view. 
It is not hard to see that 
\begin{eqnarray*}
	\frac{d}{dc}a(g_{m,c})&=&(m-1)e^{\frac12 c^2-\frac12 mc^2}\int_{c}^{\infty}(x-c)e^{-\frac{1}{2}mx^2+(m-1)cx}dx \\
	&=&(m-1)e^{-\frac12 c^2}\int_{0}^{\infty}xe^{-\frac{1}{2}mx^2-cx}dx\\
	&=&(m-1)m^{-1}e^{-\frac12 c^2}\int_{0}^{\infty}xe^{-\frac{1}{2}x^2-\sqrt{m}cx}dx
\end{eqnarray*}
Similarly
\begin{eqnarray*}
	\frac{d}{dc}b(g_{m,c})&=&(m-1)e^{\frac12 c^2-\frac12 mc^2}\int_{c}^{\infty}x^2(x-c)e^{-\frac{1}{2}mx^2+(m-1)cx}dx \\
	&=&(m-1)e^{-\frac12 c^2}\int_{0}^{\infty}(x+c)^2xe^{-\frac{1}{2}mx^2-cx}dx\\
	&=&(m-1)m^{-2}e^{-\frac12 c^2}\int_{0}^{\infty}(x+\sqrt{m}c)^2xe^{-\frac{1}{2}x^2-\sqrt{m}cx}dx
\end{eqnarray*}
So we need to consider the ratio function 
$$f_1(c)=\frac{\int_0^\infty x(x+c)^2e^{-x^2/2-cx} dx}{\int_0^\infty xe^{-x^2/2-cx} dx}, \quad c\geq 0.$$
Recall the complementary error function $\erfc(x)=\frac{2}{\sqrt{\pi}}\int_{x}^{\infty}e^{-t^2}dt$ 
Then we can calculate the derivative:
$$f_1'(c)=\frac{2 \sqrt{2 \pi } \left(c^2+1\right) e^{\frac{c^2}{2}} \text{erfc}\left(\frac{c}{\sqrt{2}}\right)-4 c}{\left(\sqrt{2 \pi } c e^{\frac{c^2}{2}} \text{erfc}\left(\frac{c}{\sqrt{2}}\right)-2\right)^2}$$

Since $\erfc$ satisfies the inequality 
\begin{equation}\label{e-erfc}
	\sqrt{\pi }e^{x^2}\erfc(x)>\frac{2}{x+\sqrt{x^2+2}}.
\end{equation}
It follows that 
\begin{equation}\label{e-f1c}
	f_1'(c)>0,
\end{equation}
for $c\geq 0$.

Similarly we can consider $$f_2(c)=\frac{\int_0^\infty x(x+c)e^{-x^2/2-cx} dx}{\int_0^\infty xe^{-x^2/2-cx} dx}.$$

We then have $$f_2'(\sqrt{2}c)=\frac{2[\pi e^{2c^2}\erfc^2(c)+2\sqrt{\pi}ce^{c^2}\erfc(c)-2]}{(2-2\sqrt{\pi}ce^{c^2}\erfc(c))^2}.$$
The numerator can be considered as a quadratic function $x^2+2cx-2$ of $x=\sqrt{\pi}e^{c^2}\erfc(c)$. And the larger root of the quadratic function is $-c+\sqrt{c^2+2}$. Then by formula \ref{e-erfc} again, we have 
\begin{equation}\label{e-f2c}
f_2'(c)>0,
\end{equation}
for $c\geq 0$, which implies that the function $f_3(c)=\frac{b_1(c)}{a(c)}$, where $b_1(c)=\int_{0}^{\infty}xe^{-g_{m,c}(x)}dx$ satisfies the following.
\begin{lem}\label{lem-f3c}
When $m<1$, we have 
\begin{equation}\label{e-f3c}
f_3'(c)<0,
\end{equation}
for $c\geq 0$.
\end{lem}

\begin{proof}
 Then $$f_3'(c)=\frac{a(c)b_1'(c)-b_1(c)a'(c)}{a^2(c)}.$$
 We have
\begin{eqnarray*}
	b_1'(c)=(m-1)m^{-3/2}e^{-\frac12 c^2}\int_{0}^{\infty}(x+\sqrt{m}c)xe^{-\frac{1}{2}x^2-\sqrt{m}cx}dx,
\end{eqnarray*}
and $$a'(c)=(m-1)m^{-1}e^{-\frac12 c^2}\int_{0}^{\infty}xe^{-\frac{1}{2}x^2-\sqrt{m}cx}dx.$$
So at $0$, we have $a(0)=\sqrt{\frac{\pi}{2m}}$, $b_1(0)=\sqrt{\frac{\pi}{2}}\frac{1}{m}$, $a'(0)=(m-1)/m$ and $b_1'(0)=\sqrt{\frac{\pi}{2}}\frac{m-1}{m^{3/2}}$. Therefore $$f_3'(0)\begin{cases}
	<0, &m<1\\
	>0, &m>1
\end{cases}.$$

Write \begin{eqnarray*}
	f_3'(c)&=&\frac{a'(c)}{a(c)}[\frac{b_1'(c)}{a'(c)}-\frac{b_1(c)}{a(c)}]\\
	&=&\frac{a'(c)}{a(c)}[\frac{1}{\sqrt{m}}f_2(\sqrt{m}c)-f_3(c)]
\end{eqnarray*}
Then formula \ref{e-f2c} implies that we always have $\frac{b_1'(c)}{a'(c)}>\frac{b_1(c)}{a(c)}$, so $f_3'(c)<0$. 
\end{proof}
Then we can control $\bar{t}$ with the following proposition:
	\begin{prop}\label{prop-tbar}
		Let $g(t)$ be a $(M_1,M_2)$-admissible function. We have 
		$$|\bar{t}(g)|\leq \sqrt{\frac{2}{\pi M_1}}(1-\frac{1}{\sqrt{m}}),$$
		where $m=\frac{M_2}{M_1}$.
	\end{prop}
	\begin{proof}
		By a dilation of $t$, we can assume $M_1=1$ and $M_2=m$. Then we show that
		the extremal case is that $g''(t)=1$ for $t\geq 0$ and $g''(t)=m$ for $t<0$.
		To see this, we assume without loss of generality that $\bar{t}(g)>0$. Then we apply operation B on the part where $t<0$ and apply operation A on the part where $t>0$, both of these operations will increase $\bar{t}(g)$. So we arrive at a function $g_1$ satisfying 
		\begin{itemize}
			\item $\int_{-\infty}^{\infty}e^{-g_1(t)}dt=\int_{-\infty}^{\infty}e^{-g(t)}dt$ and $\bar{t}(g)<\bar{t}(g_1)$.
			\item $\exists c_1>0>c_2$ such that $$g_1''(t)=\begin{cases}
				m, &0< t< c_1\\
				1, &t\geq c_1\\
				1, &c_2<t\leq 0\\
				m, &t\leq c_2
			\end{cases}. $$
		\end{itemize}
		We will denote by $f_{m,c_1,c_2}$ a function satisfying the condition in the second item.
		Then we apply the following operation: we replace $g_1=f_{m,c_1,c_2}$ with $g_2=f_{m,c_1-\epsilon,c_2+\delta }$ for some positive $\epsilon,\delta$ satisfying $$\int_{-\infty}^{\infty}e^{-g_1(t)}dt=\int_{-\infty}^{\infty}e^{-g_2(t)}dt.$$
		Clearly, this operation increases $\bar{t}(g)$.

		And we can repeat applying this operation until $c_1=0$ or $c_2=0$. In the second case where $c_2=0$, by formula \ref{e-f3c}, we can increase $\bar{t}(g)$ by replacing the obtained function with the extremal case where $g''(t)=1$ for $t\geq 0$ and $g''(t)=m$ for $t<0$. In the first case where $c_1=0$, recalling that $\bar{t}=\frac{\int_{-\infty}^{\infty}te^{-g(t)}dt}{\int_{-\infty}^{\infty}e^{-g(t)}dt}$, if we increase $c_2$ we simultaneously decrease the denominator and increase the numerator so we also increase $\bar{t}(g)$ by replacing the function with the extremal case.
		And $\bar{t}$ in this case is 
		$$\frac{1-\frac{1}{m}}{\sqrt{\frac{\pi}{2}}(1+\sqrt{\frac{1}{m}})}=\sqrt{\frac{2}{\pi}}(1-\frac{1}{\sqrt{m}}).$$

		By symmetry, we will also have a lower bound. Then notice that a dilation of $t$ makes the same dilation to $\bar{t}$, and we get the conclusion.
	\end{proof}

\begin{theorem}\label{thm-dtilde}
	Let $g(t)$ be a $(M_1,M_2)$-admissible function. Let $m=\frac{M_2}{M_1}$. Then there exist $c_1\geq c_2\geq 0$ such that 
	$$\tilde{d}(g)\geq \tilde{d}(g_{m,c_1,c_2}), $$
	and $$\tilde{d}(g)<p(m). $$
\end{theorem}
\begin{proof}

	We can assume without loss of generality that $\bar{t}\geq 0$. Recall that in the proof of theorem \ref{convex-general} we have defined operations A and B to change $\int_{0}^{\infty }t^2e^{-g(t)}dt$ while keeping $\int_{0}^{\infty }e^{-g(t)}dt$. The idea here is similar.
	For the lower bound, we first notice that if we get a function $g_1$ satisfying $\int_{0}^{\infty }e^{-g_1(t)}dt=\int_{0}^{\infty }e^{-g(t)}dt=1$ and $\int_{-\infty}^{\infty }(t-\bar{t}(g))^2e^{-g_1(t)}dt\leq \int_{-\infty}^{\infty }(t-\bar{t}(g))^2e^{-g(t)}dt$, then 
	$$\int_{-\infty}^{\infty }(t-\bar{t}(g_1))^2e^{-g_1(t)}dt\leq \int_{-\infty}^{\infty }(t-\bar{t}(g))^2e^{-g(t)}dt,$$
	since $$\int_{-\infty}^{\infty }(t-\bar{t}(g_1))^2e^{-g_1(t)}dt\leq \int_{-\infty}^{\infty }(t-a)^2e^{-g_1(t)}dt,$$
	for any $a$. 
	Then we first apply operation B to the part $t\leq 0$. By abuse of notations, we will still denote by $g$ the new function after this operation. Clearly, this will decrease $\tilde{b}(g)$. And since such an operation B will increase $\bar{t}(g)$, we can continue this operation until $g(-t)$ for $t\geq 0$ is a dilation of $g_{m,c}$ for some $c$. Then we apply operation B to the part $t\geq 0$. To show that we can also continue applying operation B, we denote by $$\sigma(g)=\inf\{x>0|g''(t)=M_2 \quad \textit{for} \quad t\geq x \}. $$ 
	Then since $\bar{t}(g)$ is the center of mass, we must have $2\bar{t}(g)<\sigma(g)$. Therefore small operation B on the part $t\geq 0$ will always decrease $\tilde{b}(g)$. Let $g_2$ be the limit function, then $g_2$ must be a dilation of $g_{m,c_1,c_2}$ for some $c_1, c_2\geq 0$.

	\

	For the upper bound, we first make a dilation to make $M_2=1$. Assume again that $\bar{t}\geq 0$. We first apply operation A on $g(t)$ for $t\leq 0$. Notice that this will increase $\int_{-\infty}^{\infty}t^2e^{-g(t)}dt$ and decrease $\bar{t}$. And as long as the new $\bar{t}\geq 0$, $\bar{t}^2$ decreases. Therefore such an operation will increase $\tilde{d}(g)$. Then we keep doing this operation A until it cannot be continued. There are two possibilities:
	\begin{itemize}
		\item[(1)]$\exists c\leq 0$ such that $g(-t)$, $t\geq 0$, is $g_{1/m,c}$;
		\item[(2)]$\bar{t}=0$. 
	\end{itemize} 
	In the second case, we can apply operation A for the two sides $t\geq 0$ and $t\leq 0$ simultaneously while keeping $\bar{t}=0$. As long as neither side is some $g_{1/m,c}$, we will continue this operation. This operation will end until one side, which we can assume is the part $t\leq 0$, is some $g_{1/m,c}$. 

	Then we keep applying operation A on the part $t\geq 0$ until the limit function $g_3$ is some $g_{1/m,c'}$. We denote the new function by $g_4$. Notice that during this process, both $\int t^2 e^{-g(t)}dt$ and $\bar{t}(g)$ increases. It is not guaranteed that $\tilde{g}$ increases.

	For $t\in \R$, $g_4=g_{1/m,c',c}$. 
    Then $\bar{t}>0$ imples that $c>c'$. Then we apply a new operation D by replacing $g_{1/m,c',c}$ with $g_{1/m,c'+\epsilon,c-\delta}$ for some small positve $\epsilon, \delta$ satisfying $$\tilde{a}(g_{1/m,c'+\epsilon,c-\delta})=\tilde{a}(g_{1/m,c',c}).$$
	Then, formula \ref{e-f1c} implies that when $\epsilon$ is small enough, we have $$\int_{-\infty}^{\infty}t^2e^{-g_{1/m,c',c}(t)}dt<\int_{-\infty}^{\infty}t^2e^{-g_{1/m,c'+\epsilon,c-\delta}}dt.$$
	So we can continue applying operation D until we get $c=c'=c_3$. Notice now that $\bar{t}(g_{1/m,c_3,c_3})=0$ and we have 
	$\int t^2 e^{-g(t)}dt<\int_{-\infty}^{\infty}t^2e^{-g_{1/m,c_3,c_3}}$ and $\int e^{-g(t)}dt=\int_{-\infty}^{\infty}e^{-g_{1/m,c_3,c_3}}$, therefore, $\tilde{g}<p(m)$.

\end{proof}
\begin{remark}
	We expect that $\tilde{d}(g)>q(m) $ also holds. However, we are not able to show this due to the complexity of the calculations.
\end{remark}
\begin{cor}\label{cor-low}
	Let $g(t)$ be a $(M_1,M_2)$-admissible function. Let $m=\frac{M_2}{M_1}$. Then we have
	$$\tilde{d}(g)\geq q(m)-\frac{4}{\pi^2}\frac{(1-\sqrt{\frac{1}{m}})^2}{(1+\sqrt{\frac{1}{m}})^2}.$$
\end{cor}
\begin{proof}
	Write $A=\int_{0}^{\infty}e^{-g(y)}dy $, $B=\int_{0}^{\infty}e^{-g(y)}dy $. We have 
	$$\tilde{d}(g)=\frac{\int_{-\infty}^{\infty}t^2e^{-g(t)}dt}{(\int_{-\infty}^{\infty}e^{-g(t)}dt)^3}-\frac{\bar{t}^2(g)}{(\int_{-\infty}^{\infty}e^{-g(t)}dt)^2}$$
	 Then since 
	$$\frac{A^3+B^3}{(A+B)^3}\geq \frac{1}{4},$$ and $$|\bar{t}|\leq \sqrt{\frac{2}{\pi M_1}}(1-\sqrt{\frac{1}{m}}),$$
	we have the conclusion.
\end{proof}
This corollary does not give us a priori lower bound for $\tilde{d}(g)$. So we need to prove the following:
\begin{cor}\label{cor-12}
	Let $g(t)$ be a $(M_1,M_2)$-admissible function. Then we have
	$$\tilde{d}(g)\geq \frac{1}{12}.$$
\end{cor}
\begin{proof}
	One can repeat the proof of theorem \ref{convex-general}, and then relax $M_1\to 0$ and $M_2\to\infty$. In order for the total mass to be 1, we arrive at the limit function $$g_{\infty}(t)=\begin{cases}
		1, &c_1<t< c_2\\
		0, &t> c_2\quad \textit{or}\quad t<c_1
	\end{cases},$$
	where $c_1<0<c_2$ and $c_2-c_1=1$. Then clearly, $\tilde{d}(g_{\infty})=\frac{1}{12}$.
\end{proof}
We end this subsection with the following proposition:
\begin{prop}\label{prop-minmax}
	For fixed $m>1$, $\tilde{d}(g_{m,c_1,c_2})$ considered as a function of $c_1\geq 0, c_2\geq 0$ attains its minimum. 
\end{prop}
\begin{proof}
	We denote by $d_m=\inf_{c_1\geq 0, c_2\geq 0} \tilde{d}(g_{m,c_1,c_2})$.
	We first notice that for $c$ large enough, the three integrals $\int_{c}^{\infty}e^{-g_{m,c}(t)}dt$, $\int_{c}^{\infty}te^{-g_{m,c}(t)}dt$ and $\int_{c}^{\infty}t^2e^{-g_{m,c}(t)}dt$ are all very small. Therefore, for $c_1$ and $c_2$ large engough, $\tilde{d}(g_{m,c_1,c_2})$ is close to $\frac{1}{2\pi}$. Since it is clear that $$d_m\leq q(m)<\frac{1}{2\pi},$$ for $m>1$, we have that for any $\epsilon_1>0$, $\exists N_1>0$ such that when $c_1\geq N_1$ and $c_2\geq N_1$, we have 
	$$\tilde{d}(g_{m,c_1,c_2})>\inf \tilde{d}(g_{m,c_1,c_2})+\epsilon_1.$$

	We claim that $\exists N_2>0$ such that for $N_1\geq c_2\geq 0$, for $c_1>N_2$ we have $$\tilde{d}(g_{m,c_1,c_2})>\tilde{d}(g_{m,c_1-\epsilon,c_2+\delta}),$$
	for some $\epsilon>0, \delta>0$. By symmetry, the same holds if we switch $c_1$ and $c_2$. It will then follow that $\tilde{d}(g_{m,c_1,c_2})$ attains its minimum. 
	To see that the claim holds one only need to notice that if $g_1(t)>g(t)$ for $t>l>0$, then $$\frac{\int_{t\geq l}t^2(e^{-g(t)}-e^{-g_1(t)})dt }{\int_{t\geq l}(e^{-g(t)}-e^{-g_1(t)})dt}>l^2.$$
	Now we apply an operation by replacing $g_{m,c_1,c_2}$ with $g_{m,c_1-\epsilon,c_2+\delta}$, where $\delta>0$ is chosen according to $\epsilon>0$ so that the total mass $$\int_{-\infty}^{\infty}e^{-g_{m,c_1-\epsilon,c_2+\delta}}dt=\int_{-\infty}^{\infty}e^{-g_{m,c_1,c_2}}dt.$$  
	Since $\bar{t}(g)$ is bounded in terms of $m$, and $$\frac{\int_{0}^{\infty}t^2(e^{-g_{m,c_2+\delta}}-e^{-g_{m,c_2}})dt }{\int_{0}^{\infty}(e^{-g_{m,c_2+\delta}}-e^{-g_{m,c_2}})dt}$$ is bounded in terms of $m$ and $N_1$,  
	it is then clear that when $N_2$ is large enough, we have 
	$$\tilde{d}(g_{m,c_1,c_2})>\tilde{d}(g_{m,c_1-\epsilon,c_2+\delta}).$$

\end{proof}

	\section{Refined estimates}
	In order to apply our estimates on $\tilde{d}(g)$, we first notice that 
	\begin{eqnarray*}
		\frac{d^2}{da^2}\lambda(a)&=&\int_0^\infty\frac{(\log x-\tilde t_a)^2dx}{h_a(x)}\\
		&=&\frac{x_{a+1}}{h_a(x_{a+1})}\int_{-\infty}^\infty (t-\tilde t_a)^2 e^{-g(t)}dt .
	\end{eqnarray*}
	We have $g(t)$ a convex function satisfying $g(t_{a+1})=g'(t_{a+1})=0$ and $g''(t)=\frac{d^2}{dt^2}\log f(x)$. Also $$\tilde t_a=\bar{t}(g),$$ in the notations of the previous subsection.
	Also since $\frac{x_{a+1}}{h_a(x_{a+1})}\int_{-\infty}^\infty e^{-g(t)}dt=1$, we have the total mass $$\int_{-\infty}^\infty e^{-g(t)}dt=\frac{h_a(x_{a+1})}{x_{a+1}}.$$
	Since $$\frac{da}{dt_a}=\frac{d^2}{dt^2}\log f(x_a),$$ we have $t_{a+1}=t_a+O(\frac{1}{a})$, therefore $x_{a+1}=x_a+O(1)$. We have 
	$\frac{d^2}{dt^2}\log h_a(x)=O(x)$ and $\frac{d}{dt}\log h_a(x_a)=0$. So $$\log h_a(x_{a+1})=\log h_a(x_a)+O(\frac{1}{a}).$$
	Therefore 
	\begin{equation}\label{e-g}
		\int_{-\infty}^\infty e^{-g(t)}dt=\frac{h_a(x_{a+1})}{x_{a+1}}=(1+O(\frac{1}{a}))\frac{h_a(x_{a})}{x_{a}}.
	\end{equation}
	
	For $\frac{d^2}{dt^2}\log f(x)$, we have $$\frac{d^2}{dt^2}\log f(x_a)=\frac{d^2}{dt^2}\log h_a(x_a)=\sum(i-a)^2\tau_{x_a}(i),$$
	where \begin{eqnarray*}
		\tau_{x_a}(i)&=&\frac{1}{h_a(x_a)}\frac{c_i}{c_a}x_a^{i-a}\\
		&=&\frac{1}{h_a(x_a)}\frac{c_{n_{x_a}}}{c_a}x_a^{n_{x_a}-a}e^{-G(i)},
	\end{eqnarray*}
	where $G(i)$ is a convex function satisfying
	$G(n_{x_a})=G'(n_{x_a})=0$ and $G''(i)=\frac{d^2}{di^2}\lambda(i)$. We will denote by $\Delta_a=\frac{c_{n_{x_a}}}{c_a}x_a^{n_{x_a}-a}$. Since $\frac{d^2}{da^2}\lambda(a)=O(\frac{1}{a})$, we can use integral to estimate the summations $\sum \tau_{x_a}(i)$ and $\sum(i-a)^2\tau_{x_a}(i)$ with relative error of size $O(\frac{1}{\sqrt{a}})$. Since $\sum \tau_{x_a}(i)=1$, we have
	\begin{equation}\label{e-gg}
		\int_{i\geq 0}e^{-G(i)}di=(1+O(\frac{1}{\sqrt{a}}))\frac{h_a(x_a)}{\Delta_a}.
	\end{equation}

	Now we are ready to prove the following theorem.
	\begin{theorem}\label{main-1}
		We have the following:
		\begin{equation} \label{eq3.9}
		\lim_{a\to \infty}\frac{h_a(x_a)}{\sqrt{a}}=\sqrt{2\pi},
		\end{equation}
		\begin{equation}\label{eq3.10}
		\lim_{a\to \infty}a\frac{d^2}{da^2}\lambda(a)=1,
		\end{equation}
		and
		\begin{equation}\label{eq3.11}
		\lim_{x\to \infty} \frac{\frac{d^2}{dt^2}\log f(x)}{x}=1
		\end{equation}
		
	\end{theorem}
	\begin{proof}
		Let $A=\limsup_{a\to \infty}\frac{h_a(x_a)}{\sqrt{a}}$, $B=\liminf_{a\to \infty}\frac{h_a(x_a)}{\sqrt{a}}$. We first assume
		that $A>B$ and draw a contradiction. For all $ \epsilon>0$, there exists an $ \alpha$ such that for $a\geq \alpha$, we have
		$$B-\epsilon<\frac{h_a(x_a)}{\sqrt{a}}<A+\epsilon.$$
		Denote $\rho(a)=\log \frac{h_a(x_a)}{\sqrt{a}}$. Let $b$ be a local minimum point of $\rho$ such that $\rho(b)<\log(B+\epsilon)$.
		Then  we have $\rho'(b)=0$ 
		and $\rho''(b)\geq 0$. We can calculate
		$$\rho'(a)=\frac{f'(x_a)}{f(x_a)}\frac{dx_a}{da}+\frac{d\lambda(a)}{da}-t_a-\frac{a}{x_a}\frac{dx_a}{da}
		-\frac{1}{2a}=\frac{d\lambda(a)}{da}-t_a-\frac{1}{2a},$$ 
		since $\frac{a}{x_a}=\frac{f'(x_a)}{f(x_a)}$. Therefore, 
		$$\rho''(a)=\frac{d^2}{da^2}\lambda(a)-\frac{1}{\frac{d^2}{dt^2}\log f(x_a)}+\frac{1}{2a^2}$$
		So we have $\tilde t_b-t_b=\frac{1}{2b}$ 
		and at $a=b$, $$\frac{d^2}{da^2}\lambda(a)-\frac{1}{\frac{d^2}{dt^2}\log f(x_a)}\geq -\frac{1}{2a^2}.$$
		Since $n_{\tilde x_a}=a$, we have $n_{x_a}=a+O(1)$ at $a=b$. So $$\Delta_b=e^{O(1/b)}.$$
		Suppose we have positve numbers $M_1,M_2,M_1',M_2'$ such that 
		\begin{eqnarray*}
			M_1&<&a\frac{d^2}{da^2}\lambda(a)<M_2\\
			M_1'&<&\frac{\frac{d^2}{dt^2}\log f(x)}{x}<M_2'
		\end{eqnarray*}
		Let $m=\frac{M_2}{M_1}$ and $m'=\frac{M'_2}{M'_1}$
		Let $$\tilde{q}(m)=\max_{c_1\geq 0,c_2\geq 0}\tilde{d}(g_{m,c_1,c_2}),$$
		whose existence is guanrenteed by proposition \ref{prop-minmax}. 
		For simplify, we write $p=p(m)$, $p'=p(m')$, $q=\tilde{q}(m)$ and $q'=\tilde{q}(m')$.
		Then \begin{eqnarray*}
			\frac{d^2}{da^2}\lambda(a)<\frac{p(B+\epsilon)^2}{a}(1+\epsilon),
		\end{eqnarray*}
		 and $$\frac{d^2}{dt^2}\log f(x_a)<p'a(B+\epsilon)^2(1+\epsilon),$$
		 for $a=b$ large enough. Therefore at $a=b$, we have $pp'(B+\epsilon)^4(1+\epsilon)^2>1-O(\frac{1}{b})$. 
		 Since $\epsilon$ is arbitrary, we get that $$B\geq (pp')^{-1/4}.$$ 
		 We can repeat this argument for a local maximum to get that
		 $$A\leq (qq')^{-1/4}.$$ 
		 So we get that for $a$ large enough, we have 
		 $$(pp')^{-1/4}-\epsilon<\frac{h_a(x_a)}{\sqrt{a}}<(qq')^{-1/4}+\epsilon.$$
		 Then by theorem \ref{thm-dtilde} and  formula \ref{e-g}, we have
		 \begin{equation}\label{e-lambda1}
			q'((pp')^{-1/4}-\epsilon)^2<a\frac{d^2}{da^2}\lambda(a)< p'((qq')^{-1/4}+\epsilon)^2,
		 \end{equation}
		 for $a$ large enough.
		And by theorem \ref{thm-dtilde} and  formula \ref{e-gg}, we have 
		\begin{equation}\label{e-Delta}
			\frac{1}{\Delta_a^2}q((pp')^{-1/4}-\epsilon)^2<\frac{\frac{d^2}{dt^2}\log f(x_a)}{x_a}<\frac{1}{\Delta_a^2}p((qq')^{-1/4}+\epsilon)^2,
		\end{equation}
		 for $a$ large enough. Since $\Delta_a\geq 1$, we have $$\frac{\frac{d^2}{dt^2}\log f(x_a)}{x_a}<p((qq')^{-1/4}+\epsilon)^2.$$
		 To better control the lower bound, we notice that by formula \ref{e-lambda1}, we have  
		 $$\int_{i\geq 0}e^{-G(i)}di>(1+O(\frac{1}{a}))\sqrt{\frac{2\pi }{ \gamma }},$$
		where $\gamma=p'((qq')^{-1/4}+\epsilon)^2$. Therefore we get 
		\begin{equation}
			\frac{2\pi q}{\gamma}<\frac{\frac{d^2}{dt^2}\log f(x_a)}{x_a}<p((qq')^{-1/4}+\epsilon)^2.
		\end{equation}

		Let $M_2=M_2'=2*10^6, M_1=M_1'=1/5000$, namely $m=m'=10^{10}$, then by theorem \ref{convex-general}  we can let $p=p'=2$ and $q=q'=\frac{1}{12}$ and start a process of iteration as follows. We let 
		$$\bar{m}=\frac{p'}{q'}\sqrt{\frac{pp'}{qq'}}+\epsilon, $$ 
		$$\bar{m}'=\frac{pp'}{2\pi q^2q'} +\epsilon $$ 
		Then let \begin{eqnarray*}
			\bar{p}&=&p(\bar{m})\\
			\bar{p}'&=&p(\bar{m}')\\
			\bar{q}&=&\tilde{q}(\bar{m})\\
			\bar{q}'&=&\tilde{q}(\bar{m})
		\end{eqnarray*}
		Then we redefine $p=\bar{p}$, $p'=\bar{p}'$, $q=\bar{q}$ and $q'=\bar{q}'$ and repeat the preceding arguments to get new bounds for $\frac{h_a(x_a)}{\sqrt{a}}$, $a\frac{d^2}{da^2}\lambda(a) $ and $\frac{\frac{d^2}{dt^2}\log f(x)}{x}$ and so on. 
		We use the software Mathematica to compute $p(m)$ and $\tilde{q}(m)$ and iterate the process 67 times to reach the point where $1<m<m'<1.01$, $0.99<M_1'<M1<M_2<M_2'<1.01$ and $0.158818<q'<q<p<p'<0.159493$. In order to use proposition \ref{p:F derivative estimate} we now use the lower bound in formula \ref{e-Delta} for $\frac{\frac{d^2}{dt^2}\log f(x_a)}{x_a}$. 
		Since $$a-n_{x_a}=\bar{t}(G)+O(\frac{1}{\sqrt{a}}\sqrt{a}),$$ by proposition \ref{prop-tbar} we have
		$$|a-n_{x_a}|^2<\frac{2.1a}{\pi }(1-\frac{1}{\sqrt{m}})^2.$$
		Therefore, since $\lambda'(n_{x_a})=\log x_a$, $$1\leq \Delta_a<e^{\frac{1}{2}\frac{2.2}{\pi}(1-\frac{1}{\sqrt{m}})^2}.$$
		Now we let $M_1=M_1'$ and $M_2=M_2'$. We let $$p=p'=p(m)$$, and let  $$q=q'=q(m)-\frac{4}{\pi^2}\frac{(1-\sqrt{\frac{1}{m}})^2}{(1+\sqrt{\frac{1}{m}})^2} ,$$ 
		and repeat arguments on local extremal points of $\frac{h_a(x_a)}{\sqrt{a}}$ to get 
		$$B\geq \frac{1}{\sqrt{p}}, \quad A\leq \frac{1}{\sqrt{q}},$$
		and  \begin{equation}\label{e-lambda2}
			q(p^{-1/2}-\epsilon)^2<a\frac{d^2}{da^2}\lambda(a)< p(q^{-1/2}+\epsilon)^2.
		 \end{equation}
		Then we have 
		\begin{equation}
			e^{-\frac{1.1}{\pi}(1-\frac{1}{\sqrt{m}})^2}q(p^{-1/2}-\epsilon)^2<\frac{\frac{d^2}{dt^2}\log f(x_a)}{x_a}<p(q^{-1/2}+\epsilon)^2.
		\end{equation}
		
		Then we let $$\bar{m}=\bar{m}'=\frac{p^2}{q^2}e^{\frac{1.1}{\pi}(1-\frac{1}{\sqrt{m}})^2}+\epsilon.$$
		Plugging in the formulas for $p$ and $q$, we get 
		$$\bar{m}=F(m)H^{-2}(m)Q(m)+\epsilon,$$
	where  $H(m)=1-\frac{4}{q(m)\pi^2}\frac{(1-\alpha)^2}{(1+\alpha)^2}$ and $Q(m)=e^{\frac{1.1}{\pi}(1-\alpha)^2}$. Denote by $I(m)=F(m)H^{-2}(m)Q(m)$. Clearly $I(1)=1$, so in order to show that the iteration will force $m\to 1$, we only need to show that $I'(m)<1$ for $1\leq m<1.01$. We have shown that $0<F'(m)<0.84$.   
When $1\leq m<1.01$, $0.99998 \leq H(m)\leq 1$. 
We have
$$H'(m)=\frac{4(1-\alpha)^2[(1+\alpha)^2q'-q(1+\alpha)\alpha^3]}{\pi^2q^2(1+\alpha)^4}-\frac{4(1-\alpha)\alpha^3}{9\pi^2(1+\alpha)^2}.$$
So 
$$-0.000067<H'(m)\leq 0.$$
Therefore 
$$0\leq (H^{-1}(m))'<0.00007.$$   
When $1\leq m<1.01$, $Q(m)<1.00001$ and $$Q'(m)=\frac{1.1}{\pi}(1-\alpha)\alpha^3e^{\frac{1.1}{\pi}(1-\alpha)^2}<0.0018.$$
So, when $1\leq m<1.01$, 
$$0<I'(m)'<0.86.$$
 So this iteration will indeed end up with the limit situation, namely $p=q=\frac{1}{2\pi}$ and 
		$M_1=M_2=1$. 
		So this gives a contradiction. 
	This proves that we must have $A=B$.

		For all  $\epsilon>0$, there exists $ N$ such that for $a\geq N$, we have
		$$A-\epsilon<\frac{h_a(x_a)}{\sqrt{a}}<A+\epsilon.$$
		Then if $\rho(a)$ is not monotone for $a$ large enough, then the same argument as above actually proves \eqref{eq3.9}, \eqref{eq3.10} and  \eqref{eq3.11}. If $\rho(a)$ is monotone, we assume, for example, 
		that $\rho'(a)=\tilde t_a-t_a-\frac{1}{2a}\geq 0$ for $a$ large enough.
		Then $\int_{a_0}^\infty \rho'(a)da<\infty$, so the asymptotic density of the set $E_T=\{a\leq T|\rho'(a)\geq \frac{1}{a} \}$
		$$\lim_{T\to \infty}\frac{\lambda(E_T)}{T}=0,$$
		where $\lambda$ is the Lebesgue measure.
		Since $\tilde t_a=t_{a+1}+O(\frac{\log a}{\sqrt{a}})$, we have $\lim_{a\to \infty}\rho'(a)=0$. 
		There are two possibilities:
		\begin{itemize}
			\item  $\rho'(a)$ is monotone, then we have $\rho'(a)=o(\frac{1}{a})$ and there is a sequence $a_j\to\infty$ so that $\rho''(a_j)=O(\frac{1}{a_j\log a_j})$; 
			\item $\rho'(a)$ is not monotone, then around a local minimum where $|\rho'(a)|<\frac{1}{a}$, we have $\rho''(a)=0$.
		\end{itemize}
		In both cases, we can apply the above argument to get the claimed estimates in the theorem.
		The case $\rho'(a)=\tilde t_a-t_a-\frac{1}{2a}< 0$ is similar.

	\end{proof}
	As a direct application, we get that $\lim_{x\to \infty}\frac{du}{dx}=1$. Then just as Proposition \ref{prop-tatilde}, we have the following:
	\begin{cor}\label{cor-ta}
		$\tilde t_a-t_{a+1}=o(\frac{1}{\sqrt{a}})$
	\end{cor}
	\begin{proof}
		With the notation used in the proof of the preceding theorem, we have $$\int_{t_{a+1}}^{\infty}(t-t_{a+1})e^{-g(t)}dt=(1+o(1))a^{-1/2}\int_{0}^{\infty}ye^{-y^2/2}dy$$ and similarly $$\int_{-\infty}^{t_{a+1}}(t-t_{a+1})e^{-g(t)}dt=-(1+o(1))a^{-1/2}\int_{0}^{\infty}ye^{-y^2/2}dy.$$ Summing up, we get the corollary.
	\end{proof}
	\begin{cor}\label{cor-f}
		For $x$ large enough, let $a(x)$ be the number satisfying $x=\tilde x_a$. Then 
		$$f(x)=(1+o(1))c_{[a]}x^{[a]}\sqrt{2\pi x},$$
		where $[a]$ is the round down of $a$.
	\end{cor}
	\begin{proof}
		Since $h_a(x)=\sqrt{2\pi a}(1+o(1))$, we have $f(x)=(1+o(1))c_{a}x^{a}\sqrt{2\pi x}$. But since $\frac{d^2}{da^2}\lambda(a)=(1+o(1))\frac{1}{a}$, we have $c_{[a]}x^{[a]}=(1+O(\frac{1}{a}))c_{a}x^{a}$.
	\end{proof}
	\begin{cor}\label{cor-dxa}
		We have
		$$\frac{d\tilde x_a}{da}=1+o(1).$$
	\end{cor}
	\begin{proof}
		Since $\log \tilde x_a=\lambda'(a)$, 
		$$\frac{1}{\tilde x_a}\frac{d\tilde x_a}{da}=\frac{d\log \tilde x_a}{da}=\lambda''(a).$$
		By corollary \ref{cor-ta}, $\frac{\tilde x_a}{a}=1$, so we get the conclusion.
	\end{proof}

	\section{Proof of the uniqueness}
	Now we prove the uniqueness part in Theorem \ref{t:main}.
	Let $f_0=\sum_{i=0}^\infty c_ix^i$, $f_1=\sum_{i=0}^\infty \bar{c}_ix^i$ be two functions both satisfying
	$$\int_{0}^{\infty}\frac{c_ix^i}{f_0}dx=\int_{0}^{\infty}\frac{\bar{c}_ix^i}{f_1}dx=\left\lbrace \begin{array}{rr}
	1-\beta,& \quad i=0\\
	1,& i>0
	\end{array}\right.  $$
	The goal is to prove that $f_1$ is a scalar multiple of $f_0$. Suppose not, we will draw a contradiction below.  As before we may define the functions for $i\in [1, \infty)$
	$$c(i)=(\int_0^\infty \frac{x^i}{f_0}dx)^{-1},$$
	$$\bar c(i)=(\int_0^\infty \frac{x^i}{f_1}dx)^{-1}.$$ Then $c(i)=c_i$ and $\bar c(i)=\bar c_i$ when $i$ is a positive integer. We also define $c(0)=c_0$ and $\bar c(0)=\bar c_0$. 
	
	\
	
	For $t\in [0,1]$, let $f_t(x)=tf_1(x)+(1-t)f_0(x)$. For $i\in \{0\}\cup [1, \infty)$ we define $a_i(t)=t\bar{c}(i)+(1-t)c(i)$, and $$b_i(t)=\int_{0}^{\infty}\frac{a_i(t)x^i}{f_t(x)}dx.$$ Then
	$$b_i'(t)=\int_{0}^{\infty}\frac{a'_i(t)f_t-a_i(t)f'_t}{f^2_t(x)}x^idx $$
	and $$b''_i(t)=\int_{0}^{\infty}\frac{2x^i}{f^3_t(x)}(a_i(t)f'_t-a_i'(t)f(t))f'_tdx.$$
	\begin{lem}
	We have for all $i\in \{0\}\cup[1, \infty)$  that $b_i'(0)<0<b_i'(1)$.
	\end{lem}
	\begin{proof}
		Notice that if we replace $f_1$ by $\lambda f_1$ for some $\lambda>0$ and let $g_t=t\lambda f_1+(1-t)f_0$, then $g_{t'}$ is a scalar multiple of $f_t$ for
	$$t'=\frac{t}{\lambda+(1-\lambda)t}.$$
	The map $t\mapsto t'$ is a bijection from $[0,1]$ to itself with a strictly positive derivative. This means that it suffices to prove $b_i'(0)<0$ for a particular choice of $\lambda$. For a given $i$ we can choose $\lambda$ so that $c(i)=\bar c(i)$. Then
	$$b''_i(t)=\int_{0}^{\infty}\frac{2x^ia_i(t)(f'_t)^2}{f^3_t(x)}dx>0.$$
	Since $b_i(0)=b_i(1)$, we have $b_i'(0)<0<b_i'(1)$. 
	
\end{proof}

 In particular, we have \begin{equation}\label{eqn-cabar2}
	\frac{\bar c(i)}{c(i)}(1-\beta \delta_{0i})<\int_{0}^{\infty}\frac{c(i)x^i}{f_0}\frac{f_1}{f_0}dx,
	\end{equation}
	and
	\begin{equation}\label{eqn-cabar}
	\frac{c(i)}{\bar c(i)}(1-\beta\delta_{0i})<\int_{0}^{\infty}\frac{\bar{c}(i)x^i}{f_1}\frac{f_0}{f_1}dx.
	\end{equation}

	From now on, we fix the normalization that $c_0=\bar{c}_0=1$.
	Then it is not hard to see the following:
	\begin{lem}\label{lem-Q}
		Consider the set $Q=\{\frac{c(i)}{\bar{c}(i)}| i\in \R, i\geq 1 \} $. Then we have
		$\sup Q\notin Q$ and $\inf Q\notin Q$.
		
	\end{lem}
	\begin{proof}
		First of all, we have $\sup Q>1$ and $\inf Q<1$. 
		Assume that the maximum $M$ of $Q$ is attained at $i=n$, then we have $\frac{f_0}{f_1}<M$ for all $x$. Then we have 
		$$M(1-\beta\delta_{0n})=\frac{c(n)}{\bar{c}(n)} (1-\beta\delta_{0n})\leq\int_{0}^{\infty}\frac{\bar{c}(n)x^n}{f_1}\frac{f_0}{f_1}dx<M(1-\beta\delta_{0n}),$$
		a contradiction. The proof for the minimum is the same.
	\end{proof}
		In the following, we will put a bar over a quantity or function to denote the corresponding quantity or function for $f_1$. 
	Denote $k_i=\frac{c(i)}{\bar{c}(i)}$. From Lemma \ref{lem-Q}, one sees that 
	$$\inf Q=\liminf_{i\to +\infty}k_i,$$
	and $$\sup Q=\limsup_{i\to +\infty}k_i.$$
	So $\liminf_{i\to +\infty}k_i\neq \limsup_{i\to +\infty}k_i$.
	Let $a_2$ (not necessarily an integer) be a local minimum point of $k_i$ as a function of $i$, satisfying that $k_i>k_{a_2}$ for all $i<a_2$. Let $a_1$ be the integer satisfying the following:
	\begin{itemize}
		\item $a_1<a_2$;
		\item $k_i<k_{a_1}$ for $a_1<i\leq a_2$; 
		\item $k_i\leq k_{a_1}$ for $i<a_1$.
	\end{itemize}
	By Lemma \ref{lem-Q}, it is clear that we can choose $a_1$, hence $a_2$, arbitrarily large. So our notations $o(\frac{1}{a_1})$ and $o(1)$ in the following make sense.

	With the notation from the preceding section, let $\lambda(i)=-\log c_i$ and $\bar{\lambda}=-\log \bar{c}_i$. 
	Let 
	$$\epsilon=\max_{a_1\leq a\leq a_2} |\lambda''(a)-\bar{\lambda}''(a)|\cdot a=o(1),$$
	then since $\lambda'(a_2)-\bar{\lambda}'(a_2)=0$,
	$$\log \frac{k_{a_1}}{k_{a_2}}\leq \frac{\epsilon}{2a_1}(a_2-a_1)^2.$$
	So $a_2-a_1\geq \sqrt{\frac{2\log \frac{k_{a_1}}{k_{a_2}}}{\epsilon}}\sqrt{a_1}$. In other words, $a_2$ is not close to $a_1$.
	
	When $f=f_0$, let $x_1$ satisfy $x_1=x'_{a_1}$, namely $\log x_1=\lambda'(a_1)$. Similarly, we define 
	$\log \bar{x}_1=\bar{\lambda}'(a_1)$, $\log \bar{x}_2=\log x_2=\lambda'(a_2)$. 
	So $\log x_1-\log \bar{x}_1=o(\frac{1}{a_1})$. Since $f_0(x_1)\approx \sqrt{2\pi a_1}c_{a_1}x_1^{a_1}$ and $f_1(x_1)\approx \sqrt{2\pi a_1}\bar{c}_{a_1}x_1^{a_1}$, we have
	$$\frac{f_0(x_1)}{f_1(x_1)}=(1+o(1))\frac{c_{a_1}}{\bar{c}_{a_1}}.$$
	And similarly
	$$\frac{f_0(x_2)}{f_1(x_2)}=(1+o(1))\frac{c_{a_2}}{\bar{c}_{a_2}}.$$
	Consider the integral $$I_1=\int_{0}^{\infty}\frac{\bar{c}_{a_1}x^{a_1}}{f_1}\frac{f_0}{f_1}dx.$$ First notice that
	$G(t)=-\log (\frac{\bar{c}_{a_1}x^{a_1}}{f_1}\frac{f_0}{f_1})$ is again a convex function of $t=\log x$ with $G''(t)=\frac{(1+o(1))}{x}$. We have $\int_{0}^{\infty}\frac{\bar{c}_{a_1}x^{a_1}}{f_1}dx=1$, we now estimate $\frac{f_0}{f_1}$.
	
	\begin{lem}\label{lem-ca1}
		For $x\leq x_2$, we have
		$$\frac{f_0(x)}{f_1(x)}<\frac{c_{a_1}}{\bar{c}_{a_1}}.$$
	\end{lem}
	\begin{proof}
		Recall that the summation $\sum_{i=0}^\infty c_ix^i$, when $x$ is large, is concentrated on the terms $c_ix^i$ for $i$ close to $a(x)$. By our choice, we have $\lambda'(a_1)-\bar{\lambda}'(a_1)=o(\frac{1}{a_1})$. So $x_1-\bar{x}_1=o(1)$. Since $\frac{d^2}{di^2}\log k_i=o(\frac{1}{i})$, we have that within a neighborhood of the form $|i-a_2|\leq \frac{\sqrt{a_2}}{o(1)}$, $$k_i\leq 1.$$
		Then we can use integral to estimate the summations $\sum c_i x_1^i$ and $\sum \bar{c}_i x_1^i$. For any $\delta>0$, we have $$\int_A^\infty e^{-x^2}dx<\epsilon \int_{A-2}^A e^{-x^2}dx,$$
		for $A$ large enough. Then it is easy to see that we can find $a_3>a_2$, so that 
		\begin{itemize}
			\item $k_i\leq 1$ for $a_2\leq i\leq a_3$;
			\item $a_3-a_2>\sqrt{\frac{-\log k_{a_2}}{\epsilon}}\sqrt{a_2}$
		\end{itemize}
		Then for any $\delta>0$, we have
		\begin{eqnarray*}
			\sum_{i\geq a_2}c_i x_1^i&<&(1+\delta)\sum_{a_3\geq i>a_2}c_i x_1^i   \\
			&<&(1+\delta)^2\sum_{a_3\geq i>a_2}\bar{c}_i x_1^i,
		\end{eqnarray*} 
		for $a_2$ large enough.
		So we have $$\frac{f_0(x_1)}{f_1(x_1)}<\frac{c_{a_1}}{\bar{c}_{a_1}}.$$
		Then since $x_2=x_2'$, we can repeat the argument to show that the conclusion holds for $x=x_2$. Then for $x<x_2$, we have that $n_x<a_2$ and $\bar{n}_x<a_2$. So the argument above still works. 
	\end{proof}
	
	Now we are ready to draw a contradiction. 

		We consider the function $\nu(x)=\frac{\bar{c}_{a_1}x^{a_1}}{f_1}\frac{f_0}{f_1}$. Firstly, we have 
		$$\frac{f_0(x_1)}{f_1(x_1)}=(1+o(1)) \frac{c_{a_1}}{\bar{c}_{a_1}}.$$
		By corollary \ref{cor-ta}, we have $\bar{x}_{a_1}'=x_1+o(\sqrt{x_1})$, namely the maximum point of the function $\frac{\bar{c}_{a_1}x^{a_1}}{f_1}$ is close to $x_1$. So we have $\bar{h}_{a_1}(x_1)=(1+o(1))\bar{h}_{a_1}(\bar{x}_{a_1}')$. 
		Let $x_\nu$ be the maximum point of $\lambda(x)$, then we have 
		$$\nu(x_\nu)<k_{a_1}\frac{1}{\bar{h}_{a_1}(\bar{x}_{a_1}')}.$$
		Therefore $$x_\lambda-x_1=o(\sqrt{x_1}).$$
		We adapt the notations in the proof of Lemma \ref{lem-ca1} and define $x_3=\bar{x}'_{a_3}$, namely $\log x_3=\bar{\lambda}'(a_3)$. Then by Corollary \ref{cor-dxa}, we have $x_3-x_2=(1+o(1))(a_3-a_2)$ and $$x_2-x_1>(1+o(1))\sqrt{\frac{2\log \frac{k_{a_1}}{k_{a_2}}}{\epsilon}}\sqrt{a_1}.$$
		Then as in the proof of Lemma \ref{lem-ca1}, for any $\delta>0$, we have
		$$\int_{x_2}^{\infty }\nu(x)dx \leq (1+\delta)\int_{x_2}^{x_3}\nu(x)dx,$$
		for $a_2$ large enough.
		
	For $x\in (x_2,x_3)$,  we can use Corollary \ref{cor-f} to estimate $\nu(x)<(1+\delta)\frac{\bar{c}_{a_1}x^{a_1}}{f_1}$. 
		Therefore, we get
		$$\int_{x_2}^{\infty }\nu(x)dx \leq (1+\delta)^2\int_{x_2}^{x_3}\frac{\bar{c}_{a_1}x^{a_1}}{f_1}dx,$$
		for $a_2$ large enough. 
		
		So we just need to take $\delta$ so that $(1+\delta)^2<k_{a_1}$, then by Lemma \ref{lem-ca1},
		$$\int_{0}^{\infty}\frac{\bar{c}_{a_1}x^{a_1}}{f_1}\frac{f_0}{f_1}dx<\frac{c_a}{\bar{c}_a}.$$
		This contradicts \eqref{eqn-cabar}.  Therefore we have completed the proof of the uniqueness part in Theorem \ref{t:main}.

	\begin{proof}[Proof of Corollary \ref{cor-main}]
		We write $e_\beta(x)=\sum c_n(\beta) x^n$. 
	If we let $f_0=e_\beta, \beta>0$ and $f_1(x)=e^x$, then assume that $\liminf_{n\to \infty}k_n\neq \limsup_{n\to \infty}k_n$, we can run the argument above again to get a contradiction. So we must have
	$\liminf_{n\to \infty}k_n=\limsup_{n\to \infty}k_n$, which may be infinity. Since $\int_{0}^{\infty}\frac{1}{f_0(x)}dx<1$, we get that $\sup Q>1$. Then if we repeat the proof of Lemma \ref{lem-Q}, we get that $\sup Q\notin Q$. So we have $\limsup_{n\to \infty}k_n=\sup Q$. Then if $\inf Q\leq 1$, then $\inf Q$ is attained for some $k_n\in Q$, and then we can repeat the proof of Lemma \ref{lem-Q} again to get a contradiction. Therefore 
	$$c_n(\beta)>\frac{1}{n!},$$ for all $n\geq 1$.
	
	We can also get $c_n(\beta)>c_n(\gamma)$ for $\gamma<\beta$, by letting $f_1=e_\gamma$. Therefore, if we let $$\bar{c}_n(\gamma)=\inf_{\beta>\gamma}c_n(\beta),$$
	then the entire function $f(x)=\sum_{i=0}^\infty \bar{c}_n(\gamma)x^n$ satisfies \eqref{e-2} with $\beta$ replaced by $\gamma$. To see this, one only needs to use again the property that for fixed $n$, for all $\epsilon>0$, There exists $ N>0$ such that $\int_{N}^{\infty}\frac{x^n}{e^x}dx<\epsilon$.
	So by the uniqueness theorem, $f(x)=e_\gamma (x)$. 
	So we have proved the upper semi-continuity of $c_n(\beta)$. The lower semi-continuity is similar. So by Dini's Theorem,  we have proved corollary \ref{cor-main}.
	\end{proof}

	\

	\bibliographystyle{plain}

	\bibliography{references}

\begin{thebibliography}{1}

\bibitem{loi-cn}
Fabrizio Cuccu and Andrea Loi.
\newblock Balanced metrics on {$\mathbb C^n$}.
\newblock {\em J. Geom. Phys.}, 57(4):1115--1123, 2007.

\bibitem{hayman1978}
WK~Hayman and I~Vincze.
\newblock A problem on entire functions.
\newblock {\em Complex Analysis and its Applications (dedicated to IN Vekua on
  his 70th Birthday). Izd. Nauka, Moscow}, pages 591--594, 1978.

\bibitem{Keller2014}
Julien Keller and Julius Ross.
\newblock A note on chow stability of the projectivization of gieseker stable
  bundles.
\newblock {\em The Journal of Geometric Analysis}, 24(3):1526--1546, Jul 2014.

\bibitem{loi-bundle}
Andrea Loi and Roberto Mossa.
\newblock Uniqueness of balanced metrics on holomorphic vector bundles.
\newblock {\em J. Geom. Phys.}, 61(1):312--316, 2011.

\bibitem{MilesWill}
Joseph Miles and Jack Williamson.
\newblock A characterization of the exponential function.
\newblock {\em Journal of the London Mathematical Society}, s2-33(1):110--116,
  1986.

\bibitem{RezaS-Duke}
Reza Seyyedali.
\newblock {Balanced metrics and Chow stability of projective bundles over
  Kähler manifolds}.
\newblock {\em Duke Mathematical Journal}, 153(3):573 -- 605, 2010.

\bibitem{Seyyedali2013}
Reza Seyyedali.
\newblock Balanced metrics and chow stability of projective bundles over
  k{\"a}hler manifolds ii.
\newblock {\em Journal of Geometric Analysis}, 23(4):1944--1975, Oct 2013.

\bibitem{sun-sun-2}
Jingzhou Sun and Song Sun.
\newblock An infinite dimensional balanced embedding problem {I}:existence.
\newblock {\em Preprint}, arxiv 2309.01097, Available at
  http://arxiv.org/abs/2309.01097.

\bibitem{2005Canonical}
Xiaowei Wang.
\newblock Canonical metrics on stable vector bundles.
\newblock {\em Communications in Analysis and Geometry}, 13(2):253--286, 2005.

\end{thebibliography}
\end{document}